%% file: main.tex
\pdfoutput=1
\documentclass[a4paper]{article}
\usepackage{ifluatex} %

\usepackage[utf8]{inputenc} %

\usepackage[ngerman,main=english]{babel}
\babeltags{german=ngerman}

\usepackage{amsmath,mathtools,amssymb,amsfonts} %
\usepackage{slashed} %

\ifluatex
   \usepackage{fontspec}
   \usepackage{unicode-math} %
\else
  \usepackage{lmodern}
  \usepackage[T1]{fontenc}
\fi

\usepackage[autostyle=true]{csquotes} %
\usepackage[final]{microtype}

\usepackage{amsthm}
\usepackage{thmtools}
\usepackage[pdfusetitle]{hyperref}

\usepackage[citestyle=alphabetic,
            bibstyle=alphabetic,
            backend=biber,
            url=true,
            urldate=comp,
            doi=true,
            isbn=false,
            giveninits=true]{biblatex}
\renewcommand*{\bibfont}{\small}
\bibliography{literature}
\AtEveryBibitem{\clearfield{month}}
\AtEveryBibitem{\clearfield{day}}
\AtEveryBibitem{%
  \ifentrytype{thesis}
    {}
    {
      \ifentrytype{online}{}
      {
      \clearfield{url}%
      \clearfield{urldate}%
      }
    }%
}

\usepackage{tikz}
\usetikzlibrary{babel}
\usetikzlibrary{cd} %
\usetikzlibrary{calc}

\usepackage[shortlabels]{enumitem}
\setlist[enumerate]{label=\upshape(\arabic*)}
\newlist{myenumi}{enumerate}{1}
\setlist[myenumi,1]{label=\upshape(\roman*)}
\newlist{myenuma}{enumerate}{1}
\setlist[myenuma,1]{label=\upshape(\alph*)}

\usepackage{xcolor}

\usepackage{todonotes}

\declaretheorem[name=Theorem,numberwithin=section]{thm}
\declaretheorem[name=Theorem, numbered=no]{thm*}
\declaretheorem[name=Lemma,numberlike=thm]{lem}
\declaretheorem[name=Lemma,numbered=no]{lem*}
\declaretheorem[name=Corollary,numberlike=thm]{cor}
\declaretheorem[name=Proposition,numberlike=thm]{prop}
\declaretheorem[name=Definition,numberlike=thm, style=definition]{defi}
\declaretheorem[name=Example, numberlike=thm, style=remark]{ex}
\declaretheorem[name=Remark, numberlike=thm, style=remark]{rem}
\declaretheorem[name=Conjecture, numberlike=thm, style=remark]{conj}

\numberwithin{equation}{section}
\allowdisplaybreaks[1]

\usepackage[noabbrev]{cleveref} %
\crefname{figure}{Figure}{Figures}
\crefname{table}{Table}{Tables}
\crefname{thm}{Theorem}{Theorems}
\crefname{lem}{Lemma}{Lemmas}
\crefname{defi}{Definition}{Definitions}
\crefname{cor}{Corollary}{Corollaries}
\crefname{prop}{Proposition}{Propositions}
\crefname{ex}{Example}{Examples}
\crefname{rem}{Remark}{Remarks}
\crefname{conj}{Conjecture}{Conjectures}
\crefname{section}{Section}{Sections}
\crefname{chapter}{Chapter}{Chapters}
\crefname{appendix}{Appendix}{Appendices}
\crefdefaultlabelformat{#2\textup{#1}#3}
\creflabelformat{enumi}{(#2#1#3)}

\usepackage{xparse} %
\NewDocumentCommand{\parensup}{m}{\textup{(}#1\textup{)}}

\usepackage[psc]{rudismacros}
\NewDocumentCommand{\Dirac}{}{\slashed{\mathfrak{D}}}

\NewDocumentCommand{\tensgr}{}{\mathbin{\widehat{\otimes}}}
\NewDocumentCommand{\boxtensgr}{}{\mathbin{\widehat{\boxtimes}}}

\NewDocumentCommand{\SpinBdl}{}{\mathfrak{S}}
\DeclareMathOperator{\clm}{c}

\DeclareMathOperator{\ind}{index}
\NewDocumentCommand{\PM}{}{\mathrm{PM}}
\DeclareMathOperator{\dist}{dist}
\DeclareMathOperator{\width}{width}

\usepackage{authblk}

\title{Band width estimates via the Dirac operator}
\author{Rudolf Zeidler\thanks{Funded by the Deutsche Forschungsgemeinschaft (DFG, German Research Foundation) under Germany's Excellence Strategy EXC 2044--390685587, Mathematics Münster: Dynamics -- Geometry -- Structure.}}
\affil{Mathematical Institute\\University of Münster, Germany\\\vspace{0.2cm}
email:~\href{mailto:math@rzeidler.eu}{math@rzeidler.eu}\\ url:~\href{https://www.rzeidler.eu}{www.rzeidler.eu}}
\hypersetup{
  colorlinks=true,
  allcolors=blue,
  pdfauthor=Rudolf Zeidler,
}
\date{}
\begin{document}

\maketitle
\begin{abstract}
  Let $M$ be a closed connected spin manifold such that its spinor Dirac operator has non-vanishing (Rosenberg) index.
  We prove that for any Riemannian metric on $V = M \times [-1,1]$ with scalar curvature bounded below by $\sigma > 0$, the distance between the boundary components of $V$ is at most $C_n/\sqrt{\sigma}$, where $C_n = \sqrt{(n-1)/{n}} \cdot C$ with $C < 8(1+\sqrt{2})$ being a universal constant.
  This verifies a conjecture of Gromov for such manifolds.
  In particular, our result applies to all high-dimensional closed simply connected manifolds $M$ which do not admit a metric of positive scalar curvature.
  We also establish a quadratic decay estimate for the scalar curvature of complete metrics on manifolds, such as $M \times \R^2$, which contain $M$ as a codimension two submanifold in a suitable way.
  Furthermore, we introduce the ``$\mathcal{KO}$-width'' of a closed manifold and deduce that infinite $\mathcal{KO}$-width is an obstruction to positive scalar curvature.
\end{abstract}
\section{Introduction}
There are two known techniques to prove that a given high-dimensional smooth manifold does not admit a metric of positive scalar curvature.
One is based on the Schrödinger--Lichnerowicz vanishing theorem~\cite{Schroedinger,Lichnerwociz:Spineurs} which implies that differential topological invariants associated to the Dirac operator on a spin manifold yield obstructions to positive scalar curvature.
The other goes back to \citeauthor{SchoenYau:HypersurfaceMethod}~\cite{SchoenYau:HypersurfaceMethod} and works by constructing (chains of) minimal hypersurfaces.
Finding candidates for suitable hypersurfaces is a homological problem and hence this yields topological obstructions.
Both methods have a very different flavour than classical metric comparison theorems for positive lower bounds on sectional curvature or Ricci curvature.
Remarkably, however, in cases where the minimal surface obstruction technique is available, Gromov~\cite{Gromov:MetricInequalitiesScalar} recently proved certain quantitative distance estimates in the presence of a lower bound on the \emph{scalar} curvature.
The goal of the present article is to establish similar results in situations where the Dirac operator method applies.

Our impetus is the following conjecture which Gromov formulated in \cite[p.~2, Question 5]{Gromov:101} and (with the sharp constant) in \cite[11.12, Conjecture C]{Gromov:MetricInequalitiesScalar}.
\begin{conj}\label{BandConjecture}
	Let \(M\) be a closed manifold of dimension \(n-1 \geq 5\) which does not admit a metric of positive scalar curvature.
	There exists a constant \(C_n < \infty\) such that  every Riemannian manifold \(V\) which is diffeomorphic to \(M \times [-1,1]\) and has scalar curvature bounded below by \(\sigma > 0\) satisfies
	\[
		\width(V) \coloneqq \dist(\partial_- V, \partial_{+}V) \leq \frac{C_n}{\sqrt{\sigma}},
	\]
	where \(\partial_{\pm} V\) denotes the boundary component corresponding to \(M \times \{\pm 1\}\).
\end{conj}
The evidence so far suggests that the optimal such constant should only depend on the dimension of \(M\).
 More ambitiously, it is conjectured to be
\begin{equation}C_n = 2 \pi \sqrt{\frac{n-1}{n}} < 2 \pi.\label{eq:OptimalConstant}\end{equation}

\begin{rem}\label{rem:OptimalConstant}
The constant \labelcref{eq:OptimalConstant} would indeed be the best possible. That is, for any manifold \(M\) of dimension \(n-1\) and \(\varepsilon >0\), there exists a Riemannian metric \(g\) on \(V = M \times [-1,1]\) with \(\scal_{g} \geq n(n-1)\) and \(\width(V, g) \geq 2 \pi /n - \varepsilon\), see \cite[p.~653, Optimality of \(2\pi/n\)]{Gromov:MetricInequalitiesScalar}.
\end{rem}

\begin{rem}\label{rem:PSConM}
If \(M\) admits a metric of positive scalar curvature \(g\), then no such constant exists because of the cylinder metric \(g \oplus \D t^2\) on \(M \times [-l,l]\) for arbitrary \(l > 0\).
\end{rem}

To put this conjecture in the context of classical Riemannian geometry, consider stronger curvature conditions for a moment.
The second variation formulas imply that, if there is a lower bound \(\sigma > 0\) on sectional curvature, a minimal geodesic has length at most \(\pi / \sqrt{\sigma}\).
More generally, this holds for a lower bound \((n-1) \sigma > 0\) on Ricci curvature.
This is what underlies the classical Bonnet--Myers theorem~\cite{Myers}.
Thus \cref{BandConjecture} can be viewed as asking for an analogue of these results for scalar curvature.

Gromov established \cref{BandConjecture} for the torus and related manifolds with the optimal constant, see~\cite[Sections~2, 4, 11.7]{Gromov:MetricInequalitiesScalar}.
For more general classes of manifolds which are approachable by the Schoen–Yau minimal hypersurface method it is proved with a slightly larger upper bound~\cite[Sections~5, 6]{Gromov:MetricInequalitiesScalar}.
However, until now, no simply connected examples have been shown to satisfy \cref{BandConjecture}  with any constant.
Note that the hypersurface method presupposes that \(\HZ^1(\pi_1 M; \Z) \neq 0\).

In the realm of spin manifolds, the most general known obstruction based on the Dirac operator is the Rosenberg index~\cite{Rosenberg:PSCNovikovI,Rosenberg:PSCNovikovII,Rosenberg:PSCNovikovIII}.
For a spin manifold \(M\) of dimension \(n-1\), it is an element \(\alpha(M) \in \KO_{n-1}(\Cstar \pi_1 M)\) of the real K-theory of the group \textCstar-algebra of its fundamental group.
Here one can use the maximal or the reduced completion of the real group ring.
Our main results apply to both cases, so we do not specify a choice.
More precisely, the Rosenberg index is the image of the fundamental class of \(M\) in spin bordism under the sequence of transformations,
\[
  \StolzBordism_{n-1}(M) \xrightarrow{\mathrm{ABS}} \KO_{n-1}(M) \xrightarrow{c} \KO_{n-1}(\Bfree \pi_1 M)  \xrightarrow{\nu} \KO_{n-1}(\Cstar \pi_1 M),
\]
where the first map is the Atiyah--Bott--Shapiro orientation, the second is induced by the classifying map of the universal covering, and the last is the analytic assembly map featuring in the strong Novikov conjecture.
If \(M\) is simply connected, then \(\alpha(M) \in \KO_{n-1}(\R) = \KO^{-n+1}\) reduces to the \(\alpha\)-invariant of \citeauthor{Hitchin:HarmonicSpinors}~\cite{Hitchin:HarmonicSpinors}.
For \(n-1 = 4k\), the latter identifies with the classical Atiyah--Singer index of the spinor Dirac operator on \(M\).

Our main result establishes \cref{BandConjecture} for all spin manifolds with non-vanishing Rosenberg index, albeit not with the conjecturally optimal upper bound.

\begin{thm}\label{BandTheorem}
	There exists a universal constant \(C < 8(1+\sqrt{2})\) such that the following holds.
  Let \(M\) be a closed spin manifold of dimension \(n-1\) with non-vanishing Rosenberg index \(\alpha(M) \in \KO_{n-1}(\Cstar\pi_1 M)\).
  Then every Riemannian manifold \(V\) which is diffeomorphic to \(M \times [-1,1]\) and has scalar curvature bounded below by \(\sigma > 0\) satisfies
	\[
		\width(V) = \dist(\partial_- V, \partial_{+}V) \leq \frac{C_n}{\sqrt{\sigma}},
  \]
  where \(C_n \coloneqq \sqrt{(n-1)/n} \cdot C < C\).
\end{thm}

As a consequence, this establishes \cref{BandConjecture} in the simply connected case:

\begin{cor}\label{CorSimplyConnected}
  There exists a universal constant \(C < 8(1+\sqrt{2})\) such that the following holds.
  Let \(M\) be a closed simply connected manifold of dimension \(n-1\geq 5\) which does not admit a metric of positive scalar curvature.
  Then every Riemannian manifold \(V\) which is diffeomorphic to \(M \times [-1,1]\) and has scalar curvature bounded below by \(\sigma > 0\) satisfies
	\[
		\width(V) = \dist(\partial_- V, \partial_{+}V) \leq \frac{C_n}{\sqrt{\sigma}},
  \]
  where \(C_n \coloneqq \sqrt{(n-1)/n} \cdot C < C\).
\end{cor}
\begin{proof}
  If \(M\) is of dimension at least five, simply connected and does not admit a metric of positive scalar curvature, then \(M\) is spin by \citeauthor{GromovLawson:Classification}~\cite{GromovLawson:Classification} and subsequently \(\alpha(M) \neq 0 \in \KO^{-n+1}\) by \citeauthor{Stolz:SimplyConnected}~\cite{Stolz:SimplyConnected}.
  Hence the statement follows from \cref{BandTheorem}.
\end{proof}

In particular, this answers the question for exotic spheres which do not admit a metric of positive scalar curvature. This was specifically asked in \cite[p.~59, Question~58]{Gromov:101}.

\begin{ex}\label{ex:Exotic}
  \Cref{BandConjecture} holds for the \emph{Hitchin spheres}~\cite[44]{Hitchin:HarmonicSpinors}.
  That is, an exotic sphere \(\Sigma\) of dimension \(8k+j\), where \(j \in \{1,2\}\), which does not admit a metric of positive scalar curvature because \(0 \neq \alpha(\Sigma) \in \KO^{-8k-j} \cong \Z/2\).
Thus, whether or not an estimate as in \cref{BandConjecture} holds depends in general on the differential structure.
\end{ex}

Our main result also implies lower bounds on principal curvatures of certain immersed submanifolds, see \cref{AsymptoticPrincCurv} below.
For instance, this yields new lower bounds for codimension one immersions of the Hitchin spheres into the Euclidean unit ball.

Moreover, \cref{BandTheorem} applies to every (area-)enlargeable spin manifold (see~\cite{HankeSchick:Enlargeable,HankeSchick:EnlargeableInfinite}), and each aspherical spin manifold whose fundamental group satisfies the strong Novikov conjecture~\cite{Rosenberg:PSCNovikovI}.

\begin{rem}\label{rem:twistedIndex}
  We expect that it is possible to extend our methods to manifolds which do not admit a spin structure themselves but whose universal covering is spin.
  For this we would use the twisted versions of the group \textCstar-algebra and the Rosenberg index which Stolz introduced, see~\cite[Section~5]{RosenbergStolz:PSCSurgeryConnections}, \cite{Stolz98Concordance}.
\end{rem}

\begin{rem}
  More generally than \cref{CorSimplyConnected}, our result proves \cref{BandConjecture} for all spin manifolds which satisfy the \emph{unstable} Gromov--Lawson--Rosenberg conjecture.
  Recall that this conjecture asserts that a spin manifold of dimension \(\geq 5\) admits a metric of positive scalar curvature if and only if its Rosenberg index vanishes, see~\cite[Conjecture~4.8]{RosenbergStolz:PSCSurgeryConnections}.
  This could be generalized by asserting that the twisted Rosenberg index mentioned in \cref{rem:twistedIndex} is the only obstruction to positive scalar curvature for manifolds whose universal covering is spin, and that totally non-spin manifolds always admit a metric of positive scalar curvature.
  However, while this has been proved in specific cases, already for spin manifolds it is known to be false in general~\cite{Schick:Counterexample}.
  Thus, on the one hand, we cannot expect to prove \cref{BandConjecture} using only our theorem.
  On the other hand, the known counterexamples to the unstable Gromov--Lawson--Rosenberg conjecture are based on the minimal hypersurface technique and hence accessible to the methods of \cite{Gromov:MetricInequalitiesScalar}.
  This means that finding candidates for counterexamples to \cref{BandConjecture} will likely require completely new obstructions to positive scalar curvature.
\end{rem}

\begin{rem}\label{ConstantRemark}
	The precise constant that our proof yields is
	\[ C = 4 \min_{\delta \in (0,1)} \left( \frac{1}{1-\delta} + \frac{2}{\sqrt{\delta}}\right).\]	Setting \(\delta = 1/2\) yields the upper bound \(8(1+\sqrt{2}) \approx 19.31\).
  Numerically, the minimum is attained at \(\delta \approx 0.4503\) with \(C \approx 19.20\).
  It is an interesting question if our method can be optimized to yield a stronger upper bound.
\end{rem}

  Our method also works for topologically non-trivial proper bands, see \cref{GeneralBandTheorem} below.
  A band is a compact manifold \(V\) together with distinguished parts \(\partial_\pm V\) of its boundary.
  For the technical definition, see \cite[Section~2]{Gromov:MetricInequalitiesScalar} or \cref{sec:BandWidth} below.
  As a consequence, one can deduce a quadratic decay theorem for the scalar curvature on \(M \times \R^2\).
  In the following theorem, we strengthen this by combining our result with a construction of \citeauthor{HankePapeSchick:CodimensionTwoIndex}~\cite{HankePapeSchick:CodimensionTwoIndex}.
  We prove a quadratic decay estimate for the scalar curvature on complete spin manifolds in the presence of a suitable codimension two submanifold.

\begin{thm}[Quadratic decay for codimension two]\label{QuadraticDecay}
  Let \(X\) be an \(n\)-dimensional complete connected Riemannian spin manifold and \(M \subset X\) a closed connected submanifold of codimension two with trivial normal bundle.
  Assume that the inclusion induces an injection \(\pi_1 M \to \pi_1 X\) and a surjection \(\pi_2 M \to \pi_2 X\).
  Moreover, suppose that \(M\) has non-vanishing Rosenberg index \(\alpha(M) \in \KO_{\ast}(\Cstar \pi_1 M)\).
  Then for every base-point \(x_0 \in M\), there exists \(R_0 \geq 0\) such that
  \[
    \min \{ \scal_X(x) \mid x \in \Ball_R(x_0) \} \leq \frac{C_n^2}{(R-R_0)^2}
  \]
  for each \(R > R_0\), where \(\scal_X\) denotes the scalar curvature function of the metric on \(X\) and \(C_n = \sqrt{(n-1)/n} \cdot C\) with \(C < 8(1+\sqrt{2}) \) being the same constant as in \cref{ConstantRemark}.
\end{thm}
\begin{ex}
  The manifold \(X = M \times \R^2\) with \(\alpha(M) \neq 0\) satisfies the hypotheses of the theorem.
\end{ex}
In particular, under the hypotheses of \cref{QuadraticDecay}, \(X\) does not admit a complete metric of uniformly positive scalar curvature.
But this already follows from the methods in \cite{HankePapeSchick:CodimensionTwoIndex}.
Hence our result can be viewed as a quantitative strengthening of the codimension two obstruction of \citeauthor{HankePapeSchick:CodimensionTwoIndex}.
The latter, in turn, was inspired by a theorem of \citeauthor{GromovLawson:PSCDiracComplete}~\cite[Theorem~7.5]{GromovLawson:PSCDiracComplete}.

Furthermore, we introduce the \emph{\(\mathcal{KO}\)-width} of a Riemannian manifold.
This is motiviated by similar notions which were introduced in \cite{Gromov:MetricInequalitiesScalar}.
Loosely speaking, the \(\mathcal{KO}\)-width of a manifold \(X\) is the supremum of widths of locally isometrically embedded bands which are spin and admit a flat bundle such that the twisted Dirac operator on the boundary components has non-vanishing index in real K-theory.
The precise definition is given in \cref{sec:KOWidth}.
Using this language, our main result implies the following.
\begin{thm}
  Let \(M\) be a closed manifold of infinite \(\mathcal{KO}\)-width.
  Then \(M\) does not admit a metric of positive scalar curvature.
\end{thm}
We observe that manifolds which satisfy the codimension two obstruction from \cite{HankePapeSchick:CodimensionTwoIndex} or the codimension one obstruction from \cite[Theorem~1.7]{Zeidler:IndexObstructionPositive} have infinite \(\mathcal{KO}\)-width.
It is a meta-conjecture of Schick~\cite[Conjecture~1.5]{Schick:ICM} that the Rosenberg index encompasses every obstruction to positive scalar curvature which is based on Dirac operator methods.
Hence we expect that for spin manifolds, infinite \(\mathcal{KO}\)-width implies the non-vanishing of the Rosenberg index.
This is the case in all the examples we mention and, as we explain in \cref{sec:KOWidth}, is implied by injectivity of the Baum--Connes assembly map via the \emph{stable} Gromov--Lawson--Rosenberg conjecture.
However, it remains an open question in general.

The article is structured as follows.
In \cref{sec:TechnicalTheorem}, we state and prove a technical theorem on which our results are based.
In \cref{sec:BandWidth}, we deduce the band width estimate and the quadratic decay theorem.
In \cref{sec:KOWidth}, we study the notion of \(\mathcal{KO}\)-width.
In the \cref{AppendixPMI,AppendixInj}, we exhibit index-theoretic results which are essentially known but not explicitly stated in the literature in the way we need them.

\paragraph{Acknowledgements.}
I would like to thank Johannes Ebert for valuable discussions, and Bernd Ammann as well as the anonymous referees for useful suggestions.

\section{The quantitative codimension one obstruction}\label{sec:TechnicalTheorem}
Our results are based on the technical \cref{TechnicalTheorem}.
It states that on a complete manifold over the real line, where the fibers admit an index-theoretic obstruction to positive scalar curvature, there is a universal scale-invariant upper bound on the length of each region with a positive lower bound on the scalar curvature.
Similar estimates as in the proof of \cref{TechnicalTheorem} are used in~\cite[Section~8]{EbertRandal-Williams:PSCCobordism} for a different purpose.

Start with a remark on our setup for real K-theory.
For technical reasons we work in the category of \emph{Real} \textCstar-algebras, that is, complex \textCstar-algebras together with an involutive conjugate linear \(\ast\)-automorphism.
This is equivalent to the category of \textCstar-algebras over the real numbers.
We refer to \cite{Schroeder:RealK} for a detailed exposition.
We will slightly abuse notation and use the symbol \enquote{\(\KO\)} for the K-theory groups of Real \(\textCstar\)-algebras.
Note that Real \textCstar-algebras also encompass spaces with an involution which feature in Atiyah's \(\mathrm{KR}\)-theory.
However, in this article we do not use involutions on the space level and KR-theory in a non-trivial way---the Real structures are only part of the coefficients.

We consider the following geometric setup.
Let \(W\) be a complete \(n\)-dimensional spin manifold together with a proper Lipschitz map \(x \colon W \to \R\).
Let \(A\) be some Real \textCstar-algebra and let \(E \to W\) be a smooth bundle of finitely generated projective Real Hilbert \(A\)-modules furnished with a metric connection.
Let \(\Dirac_{W,E}\) denote the spinor Dirac operator of \(W\) twisted by \(E\).
For expositions of the relevant background material about Dirac operators linear over \textCstar-algebras, we refer to~\cite{HankePapeSchick:CodimensionTwoIndex,Ebert:EllipticRegularityDirac}.
Associated to these data, there is the \emph{partitioned manifold index} which we will denote by \(\ind_{\PM}(\Dirac_{W, E}, x) \in \KO_{n-1}(A).\)
The partitioned manifold index theorem states the following.
If \(x\) is smooth near \(x^{-1}(a)\) for some \(a \in \R\) such that \(a\) is a regular value, then with \(M \coloneqq x^{-1}(a)\) we have the identity
\begin{equation}
	\ind(\Dirac_{M,E|_M}) = \ind_{\PM}(\Dirac_{W, E}, x) \in \KO_{n-1}(A).
  \label{eq:PMIT}
\end{equation}
We provide a quick definition of the partitioned manifold index which is suitable to our purposes.
In general, there are two approaches.
One is via the Roe algebra and the coarse index~\cite{Roe:IndexTheoryCoarseGeometry}.
The other, which we will use here, is to define it as an index of a certain Callias-type operator on the manifold \(W\) itself.
Indeed, if \(x\) is smooth with uniformly bounded gradient, then \(\ind_{\PM}(\Dirac_{W, E}, x)\) can be defined as the index of the unbounded regular Fredholm operator
\[
	B = \Dirac_{W, E} \tensgr 1 + r\ x\tensgr \epsilon,
\]
where \(\epsilon\) denotes left-multiplication by the Clifford generator of \(\Cl_{0,1}\) and  \(r > 0\) is an auxilliary constant that can be picked arbitrarily.
We consistenly work with \(\Cl_{n,0}\)-linear Dirac operators~\cite[Chapter~II,~§7]{LawsonMichelsohn:SpinGeometry}.
So \(B\) acts as an unbounded operator on \(\mathrm{L}^2(\SpinBdl_W \tensgr E \tensgr \Cl_{0,1})\), where \(\SpinBdl_W\) is the \(\Cl_{n,0}\)-linear spinor bundle, and is linear over the graded \textCstar-algebra  \(\Cl_{n,0} \tensgr A \tensgr \Cl_{0,1}\).
The operator \(B\) is indeed Fredholm because the corresponding Schrödinger-type operator
\[
  B^2 = \Dirac^2_{W,E} \tensgr 1 + r \clm(\D x) \tensgr \epsilon + r^2 x^2,
\]
is bounded below at infinity by the assumptions on \(x\).
Here \(\clm\) is the Clifford multiplication operator of the twisted Dirac bundle.
The symbol \enquote{\(\tensgr\)} refers to the graded tensor product which ensures that \(\Dirac_{W,E} \tensgr 1\) and \(1 \tensgr \epsilon\) anti-commute.
The index of \(B\) is then defined in \(\KO_0(\Cl_{n,0} \tensgr A \tensgr \Cl_{0,1}) \cong \KO_{n-1}(A)\).
For more detail on the index theorem behind \labelcref{eq:PMIT}, we refer to the Appendix~\labelcref{AppendixPMI}.

Moreover, note that the smoothness assumption on \(x\) is not really necessary.
It would be enough to have \(x\) in the Sobolev class \(\mathrm{W}^{1, \infty}_{\mathrm{loc}}\) with \(\nabla x \in \ELL^\infty\) to obtain a suitable operator \(B\).
However, in our technical arguments we only work with smooth \(x\) anyway to avoid having to discuss domain issues.
This is no restriction for our purposes because we can always approximate smoothly and the partitioned manifold index only depends on the coarse equivalence class of \(x\), see \cref{PartitionedManifoldTheorem}~\labelcref{item:coarseEq}.

After this preparation, we now state and prove our technical theorem.

\begin{thm}\label{TechnicalTheorem}
  For each \(n \in \N_{>0}\), set
  \[
    C \coloneqq 4 \cdot \min \left\{ \frac{1}{1-\delta} + \frac{2}{\sqrt{\delta}}\ \middle|\ {\delta \in (0,1)}\right\}, \quad C_n \coloneqq \sqrt{\frac{n-1}{n}}\ C.
  \]
	Let \(W\) be a complete Riemannian spin manifold of dimension \(n\).
	Let \(A\) be a unital Real \textCstar-algebra and \(E \to W\) a smooth bundle of finitely generated Real Hilbert \(A\)-modules endowed with a flat metric connection.
	Let \(x \colon W \to \R\) be a proper non-expanding \parensup{i.e.\ \(1\)-Lipschitz} map.
	Suppose that \(\ind_{\PM}(\Dirac_{W,E}, x) \neq 0\).
  Then for any interval \(I \subseteq \R\) such that the scalar curvature of \(W\) is bounded below by a contant \(\sigma > 0\) on \(x^{-1}(I)\), we have
	\[
		\operatorname{length}(I) \leq \frac{C_n}{\sqrt{\sigma}}
  \]
	\end{thm}
  \begin{rem}
    Under the same hypothesis \(\ind_{\PM}(\Dirac_{W,E}, x) \neq 0\), a much simpler estimate than in the proof of \cref{TechnicalTheorem} shows that the complete metric on \(W\) cannot have globally non-negative scalar curvature which is somewhere positive.
    This fact is essentially the content of \cite[Theorem~A]{Cecchini:CalliasTypePSC}.
    In contrast, the crucial point of our result is that a non-vanishing index excludes long regions with large scalar curvature regardless of what happens globally.
    In effect, this means that the completeness assumption is not relevant because we can always change the metric outside the region we care about to make it complete, and still obtain an estimate.
    In \cref{sec:BandWidth}, we exploit this by attaching complete cylinders to the boundary components in order to prove our main theorem.
  \end{rem}

	We start with technical preliminaries.

	\begin{lem}\label{InterpolatingFunctions}
  Let \(\varepsilon > 0\).
  There exist smooth functions \(\varphi_0, \varphi_1 \colon \R \to [0,1]\) such that
  \begin{itemize}
    \item \(\varphi_i(x) = i\) for \(x \leq 0\),
    \item \(\varphi_i(x) = 1-i\) for \(x \geq 1\),
    \item \(\|\varphi_i^\prime\|_\infty \leq \sqrt{2}+\varepsilon\),
    \item \(\|\varphi_0^2 + \varphi_1^2 - 1\|_\infty \leq \varepsilon\).
  \end{itemize}
\end{lem}
\begin{proof}
  Taking \[
  \varphi_0(x)
  = \begin{cases}
      0 & x \leq 0 \\
      \sqrt{2} x & x \in \left[0, \frac{1}{2}\right] \\
      \sqrt{1 - 2 (1-x)^2} & x \in \left[\frac{1}{2}, 1\right] \\
      1 & x \geq 1
    \end{cases}, \quad
  \varphi_1(x)
  = \begin{cases}
    1 & x \leq 0 \\
    \sqrt{1-2x^2} & x \in \left[0, \frac{1}{2}\right] \\
    \sqrt{2}(1-x) & x \in \left[\frac{1}{2}, 1\right] \\
    1 & x \geq 1
  \end{cases}
  \]
  satisfies all conditions exactly (with \(\varepsilon = 0\)) except that these functions are not smooth at \(x = 0, \frac{1}{2}, 1\).
  This can be remedied by slightly changing the functions at the cost of slightly increasing the maximum of the derivative and slightly perturbing the identity \(\varphi_0^2 + \varphi_1^2 = 1\).
  The details are left to the reader.
\end{proof}

\begin{lem}\label{BoundedBelow}
	Let \(B\) be a self-adjoint unbounded regular operator on some Hilbert \(A\)-module.
	Let \(c > 0\).
  If for every element \(u\) in the domain of \(B^2\), we have \(\|B u\| \geq c \|u\|\), then \(B\) is invertible.
\end{lem}
\begin{proof}
	Note that \(\|B^2 u\| \|u\| \geq \|\langle B^2 u \mid u \rangle\| = \| \langle B u \mid B u \rangle \| = \|Bu\|^2 \geq c^2 \|u\|^2\).
	Thus \(\|B^2 u\| \geq c^2 \|u\|\) for every element \(u\) in the domain of \(B^2\).
	Therefore, as in the proof of \cite[Proposition~1.21]{Ebert:EllipticRegularityDirac} it follows that \(0\) does not lie in the spectrum of \(B^2\).
  Thus \(0\) also does not lie in the spectrum of \(B\).
\end{proof}

Our argument will of course rely on the Schrödinger--Lichnerowicz formula.
However, in order to obtain the factor of \(\sqrt{\frac{n-1}{n}}\) in our constant, we need a slight strengthening of the naive estimate. 
The following lemma is a well-known observation which goes back to back to Friedrich~\cite{Friedrich:DiracEigenwert}.
See~\cite[Chapter~5]{SpinorialApproach} for a textbook treatment.

\begin{lem}\label{lem:Friedrich}
  Let \(W\) and \(E\) be as in the statement of of \cref{TechnicalTheorem}.
  Let \(u\) be in the domain of \(\Dirac_{W,E}^2\) and \(\sigma \in \R\) be such that the scalar curvature of \(W\) satisfies \(\scal(x) \geq \sigma\) for all \(x \in \supp(u)\).
  Then we have the estimate
  \[
    \langle \Dirac_{W,E}^2\ u \mid u \rangle \geq \frac{n \sigma}{4(n-1)} \left\langle u\ \middle|\ u \right\rangle.
  \]
\end{lem}
\begin{proof}
  For brevity set \(\Dirac \coloneqq \Dirac_{W,E}\).
  A computation involving the Cauchy--Schwarz inequality shows that we have the inequality
  \(
    \langle \Dirac u \mid \Dirac u \rangle \leq n \langle \nabla u \mid \nabla u \rangle.
  \)
  In fact, the latter even holds pointwise before integrating over \(W\), see \cite[p.~130, (5.7)]{SpinorialApproach} for the details of the computation.
  Since \(E\) is flat, we have the Schrödinger--Lichnerowicz formula \(\Dirac^2 = \nabla^\ast \nabla + \scal / 4\), and so
  \begin{align*}
    \langle \Dirac^2 u \mid u \rangle &\geq \langle \nabla^\ast \nabla u \mid u \rangle + \frac{\sigma}{4} \langle  u \mid u \rangle \\
   &= \langle \nabla u \mid \nabla u \rangle + \frac{\sigma}{4} \langle u \mid u \rangle  \\
   &\geq \frac{1}{n} \langle \Dirac u \mid \Dirac u \rangle + \frac{\sigma}{4} \langle u \mid u \rangle = \frac{1}{n} \langle \Dirac^2 u \mid u \rangle + \frac{\sigma}{4} \langle u \mid u \rangle.
  \end{align*}
  Hence \((1-1/n) \langle \Dirac^2 u \mid u \rangle \geq (\sigma / 4) \langle u \mid u \rangle\).
\end{proof}

We are now ready for the proof of the main technical theorem.

\begin{proof}[Proof of \cref{TechnicalTheorem}]
	First observe that by composing with a translation on \(\R\), we can always assume that \(I = [-l,l]\) with \(2l = \operatorname{length}(I)\).

  We then argue that it suffices to consider the case that \(x\) is smooth.
  Indeed, for each \(\varepsilon > 0\), there exists a smooth function \(\tilde{x}_\varepsilon\) such that \(\|x-\tilde{x}_\varepsilon\|_\infty < \varepsilon\) and \(\|\nabla x_{\varepsilon}\|_\infty < 1 + \varepsilon\), see \cite[Proposition 2.1]{GreeneWu:Approximations}.
  Then \(x_\varepsilon = (1+\varepsilon)^{-1}\tilde{x}_\varepsilon\) satisfies \(\|\nabla x_\varepsilon\|_\infty \leq 1\).
  Moreover, setting \(l_\varepsilon = (l-\varepsilon)/(1+\varepsilon)\), we have \(x_\varepsilon^{-1}([-l_\varepsilon, l_\varepsilon]) \subseteq x^{-1}([-l,l])\).
  If the theorem holds for each \(x_\varepsilon\), it follows that \(2 l_\varepsilon \leq C_n / \sqrt{\sigma}\) for every \(\varepsilon >0\), and so letting \(\varepsilon \to 0\) we deduce \(2 l \leq C_n / \sqrt{\sigma}\).

  So, now we assume that \(x\) is smooth with \(\|\nabla x\|_\infty \leq 1\).
  To prove the theorem, we argue by contraposition.
  Suppose that there exists \(\sigma > 0\) and an interval \(I = [-l, l] \subseteq \R\) with \(2l = \operatorname{length}(I) > {C_n}/\sqrt{\sigma}\) such that the scalar curvature is bounded below by \(\sigma\) on \(x^{-1}(I)\).
	Then fix \(\delta \in (0,1)\) such that \(C = 4 ( {1}/{(1-\delta)} + {2}/{\sqrt{\delta}})\).
  We set 
  \[ \kappa \coloneqq \frac{n \sigma}{4(n-1)}, \quad r \coloneqq \kappa(1 -\delta)> 0.\]
	In the following we write \(\Dirac = \Dirac_{W,E}\) for brevity.
	Then \(\ind_{\PM}(\Dirac, x)\) is equal to the index of the operator \(B = \Dirac \tensgr 1  + r x\ 1 \tensgr \epsilon\) which acts on \(\Ltwo(\SpinBdl_W \tensgr E \tensgr  \Cl_{0,1})\).
	We have
	\begin{equation}
		B^2 = \Dirac^2 \tensgr 1 + r \clm(\D x) \tensgr \epsilon + r^2 x^2. \label{eq:BSquared}
	\end{equation}
  We will prove that the operator \(B\) is invertible and thus has vanishing index.
  The rough idea is to estimate \(B^2\) separately on the regions \(x^{-1}(I)\) and \(x^{-1}(\R \setminus I)\).
  On the former, we will use our scalar curvature bound and \cref{lem:Friedrich}.
  On the latter, we will use \labelcref{eq:BSquared} and the fact that \(\Dirac^2\) is a non-negative operator.
  From this, we will deduce that \(B\) is bounded from below by a positive constant on each region.
  Then we use an interpolation with a suitable partition of unity to see that \(B\) is globally bounded from below by a positive constant.

 First, consider an element \(u\) in the domain of \(B^2\) such that \(\supp(u) \subseteq x^{-1}(I)\).
	Using \(x^2 \geq 0\), we obtain from \labelcref{eq:BSquared},
   \[
     \langle B^2 u \mid u \rangle \geq \langle \Dirac^2 \tensgr 1\ u \mid u \rangle + r\ \langle (\clm(\D x) \tensgr \epsilon) u \mid u \rangle \geq \langle \Dirac^2 \tensgr 1\ u \mid u \rangle- r |u|^2.
  \] 
  Here we used the notation \(|u| \coloneqq \langle u \mid u \rangle^{\frac{1}{2}} \in {\mathcal{A}}_+\), where \(\mathcal{A} \coloneqq \Cl_{n,0} \tensgr A \tensgr \Cl_{0,1}\) is our coefficient-\textCstar-algebra.
  We also used the estimate \(-\|T\| |u|^2 \leq \langle T u \mid u \rangle \) for a self-adjoint operator \(T\).
  Furthermore, the scalar curvature bound on \(x^{-1}(I)\) and \cref{lem:Friedrich} imply \(\langle \Dirac \tensgr 1\ u \mid u \rangle \geq \kappa |u|^2\) and hence
	\begin{equation}\label{eq:InsideEstimate}
    \langle B^2 u \mid u \rangle \geq \kappa |u|^2 - r |u|^2
			= \left( \kappa - r \right) |u|^2 = \delta \kappa |u|^2.
	\end{equation}

	Second, consider the case that \(\supp(u) \subseteq x^{-1}(\R \setminus (-d,d))\) for some \(0 < d \leq l\).
	Then, using \labelcref{eq:BSquared}, \(\langle B^2 u \mid u \rangle \geq r\ \langle (\clm(\D x) \tensgr \epsilon) u \mid u \rangle + r^2 d^2 |u|^2\) and hence
	\begin{equation}\label{eq:OutsideEstimate}
		 \langle B^2 u \mid u\rangle \geq -r |u|^2 + r^2 d^2 |u|^2 = \left(r^2 d^2 - r \right) |u|^2.
	\end{equation}
  We choose \(d \coloneqq (\sqrt{\kappa}(1 - \delta))^{-1}\) so that \(r^2 d^2 -r = \delta \kappa\).
  Using the assumption \(2l > C_n/\sqrt{\sigma} = C / (2 \sqrt{\kappa})\) and the definition of \(C\), we obtain \(l-d > 2 /\sqrt{\kappa \delta}\) and consequently \(\sqrt{2}/{(l-d)} < \sqrt{{\kappa \delta}/{2}}\).

  Now we will combine the estimates \labelcref{eq:InsideEstimate,eq:OutsideEstimate} to obtain a global lower bound.
  To this end, we fix a constant \(\lambda\) with
  \({\sqrt{2}}/{(l-d)} < \lambda < \sqrt{{\kappa \delta}/{2}}\).
  \Cref{InterpolatingFunctions} implies that for each \(\varepsilon > 0\) there exist smooth functions \(\psi_0, \psi_1 \colon \R \to [0,1]\) such that \(\supp(\psi_0) \subseteq [-l,l]\), \(\supp(\psi_1) \subseteq \R \setminus (-d,d)\), \(\|\psi_i^\prime\|_\infty \leq \lambda\) and \(\psi_0^2 + \psi_1^2 = 1 + \rho\) with \(\|\rho\|_\infty < \varepsilon\), see \cref{InterpolatingFigure}.
  \begin{figure}
    \begin{center}
    \begin{tikzpicture}[scale=0.9]
      \coordinate (O) at (0,0);
      \coordinate (Ou) at ($(O) + (0,1.5)$);
      \coordinate (Eu) at ($(O) + (0,2.5)$);
      \coordinate (D) at (3,0);
      \coordinate (L) at (5,0);
      \coordinate (E) at (6,0);
      \coordinate (Eg) at (5.7,0);
      \coordinate (d) at ($(O)-(D)$);
      \coordinate (l) at ($(O)-(L)$);
      \coordinate (e) at ($(O)-(E)$);
      \coordinate (eg) at ($(O)-(Eg)$);
      \draw[->] (e) -- (E) node[right] {\(x\)};
      \draw[->] (O) -- (Eu);
      \draw[red, thick, dotted] (e) -- (eg);
      \draw[red, thick] (eg) -- (l) sin ($(d)+(Ou)$) node[above] {\(\psi_0\)} -- ($(D)+(Ou)$) cos (L) -- (Eg);
      \draw[red, thick, dotted] (Eg) -- (E);
      \draw[blue, thick, dotted] ($(e) + (Ou)$) -- ($(eg) + (Ou)$);
      \draw[blue, thick] ($(eg) + (Ou)$) -- ($(l)+(Ou)$) node[above] {\(\psi_1\)} cos (d) -- (D) sin ($(L)+(Ou)$) -- ($(Eg) + (Ou)$);
      \draw[blue, thick, dotted] ($(Eg) + (Ou)$) -- ($(E) + (Ou)$);
      \node[below] at (l) {\(-l\)};
      \node[below] at (d) {\(-d\)};
      \node[below] at (L) {\(l\)};
      \node[below] at (D) {\(d\)};
      \node[above right] at (O) {\(0\)};
      \node[above right] at (Ou) {\(1\)};
      \node[above] at ($(D) !.5! (L) + (Ou)$) {\(\|\nabla \psi_i\|_\infty \lessapprox \sqrt{2} \cdot \frac{1}{l-d}\)};
    \end{tikzpicture}
    \end{center}
    \caption{Schematic plot of the interpolating functions \(\psi_i\)}
    \label{InterpolatingFigure}
  \end{figure}
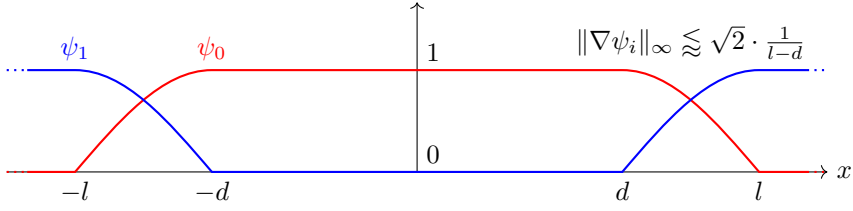
  By a slight abuse of notation, we will also write \(\psi_i\) and \(\rho\) for the functions  on \(W\) defined by \(\psi_i \circ x\) and \(\rho \circ x\), respectively.
  Then
  \begin{align}
    \langle B^2 u \mid u \rangle &= \langle B u \mid B u \rangle = \langle \psi_0^2 B u \mid B u \rangle + \langle \psi_1^2 B u \mid B u \rangle - \langle \rho B u \mid B u\rangle \nonumber \\
    &= | \psi_0 B u |^2 + | \psi_1 B u |^2 - \langle \rho B u \mid B u\rangle. \label{eq:Separation}
  \end{align}
  In the following, we estimate \(|\psi_i B u |^2\) using that \(\psi_i B = [\psi_i, \Dirac] \tensgr~1 + B \psi_i\).
  \begin{align*}
    | \psi_i B u |^2 &= \langle [\psi_i, \Dirac] \tensgr 1\ u \mid \psi_i B u \rangle
      + \langle B \psi_i u \mid \psi_i B u \rangle \\
      &= \langle [\psi_i, \Dirac] \tensgr 1\
       u \mid \psi_i B u \rangle
        + \langle B \psi_i u \mid [\psi_i, \Dirac] \tensgr 1\ u \rangle
        + \langle B \psi_i u \mid  B \psi_i u \rangle \\
        &= \langle [\psi_i, \Dirac]\tensgr 1\ u \mid \psi_i B u \rangle
          + \langle \psi_i B  u \mid [\psi_i, \Dirac] \tensgr 1\ u \rangle\\
          &\qquad+ \langle [\Dirac, \psi_i] \tensgr 1\ u \mid [\psi_i, \Dirac] \tensgr 1\ u \rangle
          + \langle B \psi_i u \mid  B \psi_i u \rangle \\
        &\geq - 2 \lambda \|u\| \|\psi_i B u\| - \lambda^2 |u|^2 +  | B \psi_i u |^2\\
        &\geq | B \psi_i u |^2 - \lambda^2 |u|^2 - 2 \lambda \|u\| \|B u\|.
  \end{align*}
  Here we used that \(\|[\psi_i, \Dirac]\| \leq \|\psi_i^\prime\|_\infty \leq \lambda\) and we identify \(\R\) with \(\R \cdot 1_{\mathcal{A}} \subset {\mathcal{A}}_{\mathrm{sa}}\).
  Together with \labelcref{eq:InsideEstimate,eq:OutsideEstimate} this implies
  \begin{equation}
  | \psi_i B u |^2 \geq \delta \kappa |\psi_i u|^2 - \lambda^2 |u|^2 - 2 \lambda \|u\| \|B u\|.
\end{equation}
  Returning to \labelcref{eq:Separation}, we obtain
  \begin{align*}
    \langle B^2 u \mid u \rangle &= | \psi_0 B u |^2 + | \psi_1 B u |^2 -\langle \rho B u \mid B u \rangle \\
    &\geq \delta \kappa \left( |\psi_0 u|^2 + |\psi_1 u|^2 \right) - 2 \lambda^2 |u|^2  - 4 \lambda \|u\| \|B u\| - \langle \rho B u \mid B u \rangle\\
    &= \delta\kappa (|u|^2 + \langle \rho u \mid u\rangle) - 2 \lambda^2 |u|^2  - 4 \lambda \|u\| \|B u\| - \langle \rho B u \mid B u \rangle\\
    &\geq \delta\kappa |u|^2  - (\varepsilon \delta\kappa + 2 \lambda^2) |u|^2  - 4 \lambda \|u\| \|B u\| - \varepsilon \langle B^2 u \mid u \rangle.
  \end{align*}
  Thus
  \[
   (1+ \varepsilon) \| Bu\|^2  \geq \delta \kappa \|u\|^2 - (\varepsilon \delta \kappa+2 \lambda^2) \|u\|^2  - 4 \lambda \|u\| \|B u\|.
  \]
  Since \(\varepsilon > 0\) was chosen arbitrarily (and independently of \(\delta\), \(\kappa\) and \(\lambda\)), we conclude that for any element \(u\) in the domain of \(B^2\), we have the estimate
  \[
    \| Bu\|^2  \geq \delta \kappa \|u\|^2 - 2 \lambda^2 \|u\|^2  - 4 \lambda \|u\| \|B u\|.
  \]
  Completing the square, it follows that
  \[
    \left( \|Bu\| + 2 \lambda \|u\|\right)^2 \geq (\delta \kappa + 2 \lambda^2) \|u\|^2
  \]
  and hence
  \[
    \|Bu\| \geq \left(\sqrt{\delta \kappa + 2 \lambda^2} - 2 \lambda \right)\|u\|.
  \]
  Finally, \(\lambda < \sqrt{\kappa\delta/2}\) implies \(\delta \kappa > 2 \lambda^2\) and so
  \[
    \sqrt{\delta \kappa + 2 \lambda^2} - 2 \lambda  > 0.
  \]
  Hence \cref{BoundedBelow} implies that \(B\) is invertible and thus \(\ind_{\PM}(\Dirac, x)=0\).
\end{proof}
\begin{rem}
  The term \(2 / \sqrt{\delta}\) in the definition of the constant \(C\) is precisely the cost of our interpolation between the positive scalar curvature region and its complement.
  If it were not there, we could take \(\delta > 0\) to be arbitrarily small, and we would obtain an upper bound of \(4/\sqrt{\sigma}\).
  This would be \emph{too good} according to \cref{rem:OptimalConstant}.
  Hence it is clear that the interpolation must come at some cost, but it is concievable that with more care our estimate can be improved to yield something closer to the optimal upper bound.
\end{rem}

\section{Band width and quadratic decay}\label{sec:BandWidth}
We start with a more general version of \cref{BandTheorem}.
We use the language from \cite[Section~2]{Gromov:MetricInequalitiesScalar}.
A \emph{band} is a manifold \(V\) with two distinguished subsets \(\partial_\pm V\) of the boundary \(\partial V\).
It is called \emph{proper} if each \(\partial_\pm V\) is a union of connected components of the boundary and \(\partial V = \partial_- V \sqcup \partial_+ V\).
If \(V\) is a Riemannian manifold, then we define \(\width(V) \coloneqq \dist(\partial_- V, \partial_+ V)\), the infimum of lengths of curves from a point in \(\partial_-V\) to a point in \(\partial_+V\).

\begin{thm}\label{GeneralBandTheorem}
  Let \(V\) be an \(n\)-dimensional compact proper band which is a Riemannian spin manifold.
  Let \(A\) be a Real unital \textCstar-algebra and \(E \to W\) a smooth bundle of finitely generated projective Hilbert \(A\)-modules endowed with a flat metric connection.
  Suppose that the index of the Dirac operator on \(\partial_-V\) twisted by \(E|_{\partial_- V}\) does not vanish in \(\KO_{n-1}(A)\).
  If the scalar curvature of \(V\) is bounded below by \(\sigma > 0\), then
  \[
		\width(V) \leq \frac{C_n}{\sqrt{\sigma}},
	\]
  where  \(C_n = \sqrt{(n-1)/n} \cdot C\) is the constant from \cref{TechnicalTheorem}.
\end{thm}
Observe that by bordism invariance---which, incidentally, follows from \labelcref{eq:PMIT}---the index on \(\partial_-V\) is the same as the one on \(\partial_+V\).

To see how this theorem relates to \cref{BandTheorem}, note that on any connected space \(X\), there is the Mishchenko line bundle
\[
  \mathcal{L}_{X} \coloneqq \tilde{X} \times_{\pi_1 X} \Cstar(\pi_1 X)
\]
which is the flat bundle of Hilbert \(\Cstar(\pi_1 X)\)-modules associated to the representation of \(\pi_1 X\) on its group \textCstar-algebra by left-multiplication.
In the case of a closed spin manifold \(M\), the \emph{Rosenberg index} \(\alpha(M)\) is the index of the Dirac operator twisted by the Mishchenko bundle, \(\ind(\Dirac_{M, \mathcal{L}_M})\).
Thus \cref{BandTheorem} follows from \cref{GeneralBandTheorem} by taking \(E = \mathcal{L}_V\) to be the Mishchenko line bundle of \(V = M \times [-1,1]\) because the inclusion of the boundary components induces an isomorphism on \(\pi_1\).
\begin{proof}[Proof of \cref{GeneralBandTheorem}]
  \begin{figure}
     \def\svgwidth{\textwidth}
      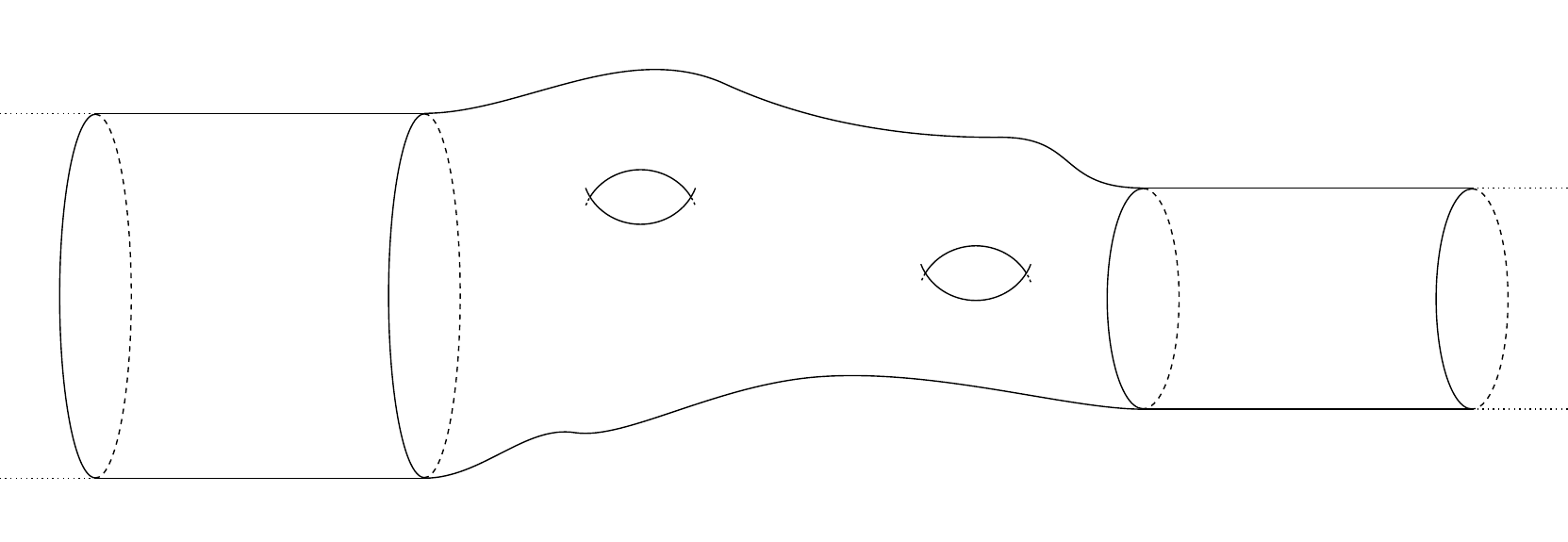
      \caption{Attaching cylinders to the boundary}
      \label{AttachingFigure}
  \end{figure}
  We construct a manifold \(W\) out of \(V\) by attaching infinite cylinders to each distinguished part of the boundary as in \cref{AttachingFigure}.
  That is,
  \[
  W \coloneqq (-\infty, -1] \times \partial_-\cup_{\partial_-} V \cup_{\partial_+} \partial_+ \times [1,\infty),
  \]
  where we used the short-hand \(\partial_\pm \coloneqq \partial_\pm V\).
  Furthermore, we set
  \[
  W_- \coloneqq (-\infty, -1] \times \partial_- \quad \text{and} \quad
  W_+ \coloneqq  V \cup_{\partial_+} \partial_+ \times [1,\infty)
  \]
  so that \(W = W_- \cup_{\partial_-} W_+\).
  We can extend the flat bundle \(E\) along the cylinders to obtain a flat bundle \(E_W\) on \(W\).

	Let \(g_V\) denote the Riemannian metric on \(V\).
	Then we fix a complete Riemannian metric \(g_W\) on \(W\) which restricts to \(g_V\) on \(V\).
	Let \(d_V\) be the length metric on \(V\) induced by \(g_V\), and \(d_W\) be the length metric on \(W\) induced by \(g_W\).
	Let \(x \colon W \to \R\) be the function which assigns to a point \(p \in W\) the signed distance to the submanifold \(\partial_-\) with respect to \(d_W\).
	That is, \(x(p) = \pm \min_{y \in \partial_-} d_{W}(y, p)\), where we use the positive sign if \(p \in W_+\) and otherwise the negative sign.
	Then \(x\) is a proper \(1\)-Lipschitz function.
	By \labelcref{eq:PMIT}, we have \(\ind_{\PM}(\Dirac_{W, E_W}, x) = \ind(\Dirac_{\partial_-, E|_{\partial_-}}) \in \KO_{n-1}(A)\) which is non-zero by assumption.

	Finally, set \(l = \width(V) = \dist_{V}(\partial_-, \partial_+)\).
	Let \(p \in x^{-1}([0,l])\).
	Then there exists a curve inside \(W_+\) of length at most \(l\) connecting \(p\) to a point in \(\partial_-\).
	This curve will eventually lie inside \(V\).
	Hence \(p \in V\) because otherwise there would be a curve inside \(V\) of length smaller than \(l\) connecting a point in \(\partial_+\) to a point in \(\partial_-\).
	Therefore we have proved that \(x^{-1}([0,l]) \subseteq V\).
	In particular, by assumption, the scalar curvature of \(W\) is bounded below by \(\sigma\) on \(x^{-1}([0,l])\).
	Thus \cref{TechnicalTheorem} implies that \(\width(V) = l \leq C_n/\sqrt{\sigma}\).
\end{proof}

Next, we turn to the quadratic decay result, \cref{QuadraticDecay}.
We start with the statement of a technical lemma taken from \cite[Theorem~4.3]{HankePapeSchick:CodimensionTwoIndex}.
\begin{lem}\label{SplitMono}
  Let \(X\) be a manifold and \(M \subset X\) be a submanifold, both connected and without boundary.
  Suppose that \(M\) has codimension two with trivial normal bundle and the inclusion \(M \hookrightarrow X\) is \(2\)-connected.
  Let \(W\) be the manifold with boundary which is obtained from \(X\) by deleting an open tubular neighborhood of \(M\).
  Then there exists a homomorphism \(r \colon \pi_1(W) \to \pi_1(M \times \Sphere^1)\) which is a retraction to the map \(\pi_1 (M \times \Sphere^1) \to \pi_1(W)\) induced by the inclusion \(M \times \Sphere^1 = \partial W \hookrightarrow W\).
\end{lem}
Note that in loc.~cit.\ the stronger assumption \(\pi_2 X = 0\) is required.
But this is unnecessary without any change to the proof.
The statement with the hypothesis as in \cref{SplitMono} has also been reproduced in the author's thesis~\cite[Lemma~4.1.4]{Zeidler:SecondaryLargescaleIndex}.

\begin{proof}[Proof of \cref{QuadraticDecay}]
  In the first part of the proof, we follow the strategy of \cite{HankePapeSchick:CodimensionTwoIndex}.
  Let \(\bar{X} \to X\) be the connected Riemannian covering of \(X\) with \(\pi_1 \bar{X} = \pi_1 M\).
  Fix a base-point \(\bar{x}_0 \in \bar{X}\) which lies over the chosen base-point \(x_0 \in M \subset X\).
  The embedding \(M \subset X\) uniqely lifts to an embedding \(M \hookrightarrow \bar{X}\) taking \(x_0\) to \(\bar{x}_0\).
  By assumption, the embedding \(M \subset \bar{X}\) satisfies the hypotheses of \cref{SplitMono}.
  Let \(W\) be the manifold obtained from deleting a tubular neighborhood of \(M\) inside \(\bar{X}\).
  Then \(\partial W = M \times \Sphere^1\) with Rosenberg index \(\alpha(\partial W) = \alpha(M) \times \alpha(\Sphere^1) \neq 0\), see~\cref{AppendixInj}.
  Using the homomorphism \(r \colon \pi_1 W \to \pi_1(M \times \Sphere^1)\) from \cref{SplitMono}, we can extend the Mishchenko bundle of \(M \times \Sphere^1\) to a flat bundle \(E = \tilde{W} \times_{\pi_1 W} \Cstar(\pi_1 M \times \Z)\) over all of \(W\).

  The next part is essentially the same as the proof of the quadratic decay results outlined in \cite{Gromov:MetricInequalitiesScalar}.
  Observe that for each \(l > 0\) and \(\varepsilon > 0\), the \(l\)-neighborhood \(\Nbh_l(\partial W \subset W)\) of \(\partial W\) in \(W\) contains a proper band \(V\) with \(\partial_- V = \partial W\) and \(\width(V) \geq l - \varepsilon\).
  Essentially, we take \(V = \Nbh_l(\partial W \subset W)\); the auxilliary \(\varepsilon\) simply allows for some wiggle-room in cases where the boundary of \(\Nbh_l(\partial W \subset W)\) is not smooth.

  To prove the result, we claim that \(\inf_{\Nbh_l(\partial W \subset W)} \scal_{\bar{X}} \leq C_n^2/l^2\).
  Here we only need to consider the case \(\sigma \coloneqq \inf_{\Nbh_l(\partial W \subset W)} \scal_{\bar{X}} > 0\).
  Then it follows from the previous paragraph and \cref{GeneralBandTheorem} that \(l-\varepsilon \leq C_n / \sqrt{\sigma}\) for any \(\varepsilon >0\).
  Therefore \(\sigma \leq C_n^2 / l^2\), as claimed.
  Finally, we let \(R_0\) be the diameter of the tubular neighborhood that was deleted from \(\bar{X}\) to obtain \(W\).
  Then we have \(\Nbh_{R-R_0}{(\partial W \subset W}) \subseteq \Ball_R(\bar{x}_0) \) for every \(R > 0\), where \(\Ball_R(\bar{x}_0)\) denotes the \(R\)-ball in \(\bar{X}\).
  Hence
  \[
    \min_{\Ball_R(x_0)} \scal_{X} = \min_{\Ball_R(\bar{x}_0)} \scal_{\bar{X}} \leq \inf_{\Nbh_{R-R_0}{(\partial W})} \scal_{\bar{X}} \leq \frac{C_n^2}{(R-R_0)^2}.\qedhere
  \]
\end{proof}

\section{Infinite \texorpdfstring{\(\mathcal{KO}\)}{KO}-width}\label{sec:KOWidth}
In this section, we introduce the \emph{\(\mathcal{KO}\)-width} for a Riemannian manifold based on the bands that appear in \cref{GeneralBandTheorem}.
This is inspired by similar notions of width which were studied in \cite{Gromov:MetricInequalitiesScalar} in terms of locally isometrically immersed bands of various types.
Start with a general definition.
\begin{defi}
  Let \(\mathcal{V}\) be some class of compact proper bands.
The \emph{\(\mathcal{V}\)-width} of a Riemannian manifold \(X\), denoted by \(\width_{\mathcal{V}}(X)\), is the supremum of all real numbers \(l > 0\) such that there exists a band \(V \in \mathcal{V}\) of the same dimension as \(X\) together with an immersion \(V \looparrowright X\) such that \(\width(V) \geq l\), where \(V\) is endowed with the Riemannian metric induced from \(X\).
If no such \(l > 0\) exists, we let \(\width_{\mathcal{V}}(X) = 0\) by convention.
\end{defi}

If \(X\) is compact, then whether or not \(\width_{\mathcal{V}}(X)\) is finite does not depend on the particular choice of Riemannian metric on \(X\).
Thus we make the following definition.
\begin{defi}
  Let \(M\) be a closed smooth manifold.
  We say that \(M\) has \emph{infinite \(\mathcal{V}\)-width} if \(\width_{\mathcal{V}}(M) = \infty\) for some (and hence any) choice of Riemannian metric on \(M\).
\end{defi}

Motivated by \cref{GeneralBandTheorem}, we introduce the following class \(\mathcal{KO}\).

\begin{defi}
  Let \(\mathcal{KO}\) be the class of compact proper bands \(V\) such that the following holds.
  \begin{itemize}
    \item There exists a Real unital \textCstar-algebra \(A\) and \(E \to V\) a smooth bundle of finitely generated Hilbert \(A\)-modules endowed with a flat metric connection.
    \item The band \(V\) is a spin manifold and the index of the Dirac operator on \(\partial_-V\) twisted by \(E|_{\partial_- V}\) does not vanish in \(\KO_{n-1}(A)\).
  \end{itemize}

  Similarly, we can define the class \(\mathcal{K}\), where we take the same definition but work with complex  \textCstar-algebras and K-theory instead.
  We obtain the class \(\mathcal{K}_{\Q}\) by insisting that the index rationally does not vanish.
  Note that \(\mathcal{K}_{\Q} \subset \mathcal{K} \subset \mathcal{KO}\).
\end{defi}

The next theorem is a reformulation of \cref{GeneralBandTheorem}.
The constants \(C\) and \(C_n = \sqrt{(n-1)/n} \cdot C\) in the following are understood to be the same as in \cref{TechnicalTheorem}.

\begin{thm}\label{InfiniteKOWidthObstruction}
  If \(X\) is an \(n\)-dimensional Riemannian manifold with scalar curvature bounded below by \(\sigma > 0\), then \(\width_{\mathcal{KO}}(X) \leq C_n / \sqrt{\sigma} < \infty\).
  In particular, if \(M\) is a closed manifold of infinite \(\mathcal{KO}\)-width, then \(M\) does not admit a metric of positive scalar curvature.
\end{thm}

The scalar curvature of the standard round sphere yields interesting bounds on the \(\mathcal{KO}\)-width of both the sphere and the Euclidean ball.
\begin{cor}\label{KOWidthBall}
  The \(\mathcal{KO}\)-width of both the round sphere and the Euclidean ball of radius \(r\) in dimension \(n\) is at most \(r C/ n\).
\end{cor}
\begin{proof}
   The sphere of radius \(r\) has scalar curvature \(n(n-1) / r^2\), so the case of the sphere is an immediate consequence of \cref{InfiniteKOWidthObstruction}.
   To see the case of the ball, consider the standard projection \(\Psi \colon \Sphere_{r,+}^n \to \Ball_r^n\) of the upper hemisphere onto the ball.
  Let \(g_{\Sphere_r^n}\) denote the round metric on the sphere and \(g_{\R^n}\) the standard Euclidiean metric on the ball.
   Then \(\Psi \colon (\Sphere_{r,+}^n, g_{\Sphere_r^n}) \to (\Ball_r^n, g_{\R^n})\) is a diffeomorphism and non-expanding.
   Now let \(b \colon V \looparrowright \Ball_r^n\) be a band in the ball with \(V \in \mathcal{KO}\).
   Then \(\Psi^{-1} \circ b \colon V \looparrowright \Sphere_{r,+}^n \subset \Sphere_r^n\) is a \(\mathcal{KO}\)-band in the sphere and hence \(\width(V, (\Psi^{-1} \circ b)^\ast g_{\Sphere_r^n}) \leq r C / n\).
   Since \(\Psi\) is non-expanding,
   \begin{align*}
    \width(V, b^\ast g_{\R^n}) &= \width(V, (\Psi \circ \Psi^{-1} \circ b)^\ast g_{\R^n})\\
    &\leq \width(V, (\Psi^{-1} \circ b)^\ast g_{\Sphere_r^n}) \leq \frac{r C}{n}. \qedhere
  \end{align*}

\end{proof}
This leads to a lower bound on the maximal principal curvatures of an immersed submanifold in the unit ball.
We focus on the asymptotic behaviour for high dimensions.
\begin{cor}\label{AsymptoticPrincCurv}
  Let \(M\) be an \((n-1)\)-dimensional closed spin manifold with non-vanishing Rosenberg index which is immersed into the \(n\)-dimensional unit ball.
  Then, asymptotically as \(n \to \infty\), the maximum of the absolute values of its principal curvatures is at least \(2 n / C\).
\end{cor}
\begin{proof}
  Let \(\varepsilon > 0\).
  Assume that \(n\) is sufficiently large such that
  \[ \delta \coloneqq \frac{(1+\varepsilon) C}{n} < 2 \varepsilon.\]
  We then claim that the maximum of the principal curvatures of \(M \looparrowright \Ball_1^n(0)\) must be at least \(2 \delta^{-1}\).
  Indeed, if not, the normal exponential map \(\exp^\perp \colon M \times \R \to \R^n\) was a local diffeomorphism on \(M \times [-l, l]\) for some \(\varepsilon > l > \delta /2\).
  Since \(\exp^\perp(M \times [-l, l])\) is contained in the \((1+\varepsilon)\)-ball, \cref{KOWidthBall} would imply \(2l \leq \delta\), a contradiction.
  Since \(2 \delta^{-1} = 2n / ((1+\varepsilon)C)\), this proves the claim.
\end{proof}
In particular, this applies to the Hitchin spheres mentioned in \cref{ex:Exotic}.
Note that every exotic sphere admits a codimension one immersion into Euclidean space.
Thus \cref{AsymptoticPrincCurv} improves Gromov's observation~\cite[668]{Gromov:MetricInequalitiesScalar} that the maximum of the principal curvatures must grow at least proportionally to \(\sqrt{n}\) in this situation.

We now turn to examples of manifolds of \emph{infinite} \(\mathcal{KO}\)-width.

\begin{ex}
  Let \(M\) be a closed spin manifold with non-vanishing Rosenberg index.
  Then \(M \times \Sphere^1\) has infinite \(\mathcal{KO}\)-width.
  This follows from the immersion \(M \times \R \looparrowright M \times \Sphere^1\) induced by the universal covering of \(\Sphere^1\).
\end{ex}

More generally:

\begin{ex}
  Let \(M\) be a closed spin manifold which admits a codimension one submanifold \(N \subset M\) with trivial normal bundle.
  Let \(\pi \leq \pi_1 M\) be the image of the homomorphism \(\pi_1 N \to \pi_1 M\) induced by the inclusion.
  Suppose that the image of the Rosenberg index of \(N\) in \(\KO_{n-1}(\Cstar \pi)\) does not vanish.
  Then the covering \(\bar{M} \to M\) with \(\pi_1 \bar{M} = \pi\) contains \(N\) as a separating hypersurface which together with the assumption implies that \(M\) has infinite \(\mathcal{KO}\)-width.
  It can also be shown that the Rosenberg index of these examples is non-zero, see \cite[Theorem~1.4]{NitscheSchickZeidler:Transfer}, \cite[Theorem~1.7]{Zeidler:IndexObstructionPositive}.
\end{ex}

\begin{ex}
  Let \(X\) satisfy the hypotheses of \cref{QuadraticDecay}.
  That is, \(X\) is a complete and connected Riemannian spin manifold; there exists a closed connected submanifold \(M \subseteq X\) of codimension two with trivial normal bundle; the inclusion induces \(\pi_1 M \hookrightarrow \pi_1 X\), \(\pi_2 M \twoheadrightarrow \pi_2 X\); the Rosenberg index of \(M\) is non-zero.
  Then \(\width_{\mathcal{KO}}(X) = \infty\) because, using the notation from the proof of \cref{QuadraticDecay} in \cref{sec:BandWidth}, there are arbitrarily wide \(\mathcal{KO}\)-bands \(V \hookrightarrow W \hookrightarrow \bar{X} \looparrowright X\).

  In particular, a closed manifold which satisfies the conditions of the codimension two obstruction of \citeauthor{HankePapeSchick:CodimensionTwoIndex}~\cite{HankePapeSchick:CodimensionTwoIndex} has infinite \(\mathcal{KO}\)-width.
  For this example Kubota~\cite{Kubota:RelativeMishchenko,KubotaSchick:Codim2} recently proved that the maximal Rosenberg index does not vanish.
\end{ex}

With a simple product construction, one can produce further examples of manifolds with infinite \(\mathcal{KO}\)-width.

\begin{prop}\label{ProductInfiniteWidth}
  Let \(N\) be a closed spin manifold such that the Rosenberg index \(\alpha(N) \in \KO_n(\Cstar\pi_1 N)\) induces an injective map
  \begin{equation}
    \KO_{\ast}(A) \xrightarrow{\blank \times \alpha(N)} \KO_{\ast+n}(A \otimes \Cstar\pi_1 N) \label{ExtProductMap}
  \end{equation}
  for every unital Real \(\Cstar\)-algebra \(A\).
  Then, if a closed manifold \(M\) has infinite \(\mathcal{KO}\)-width, so does \(M \times N\).
\end{prop}
\begin{proof}
  We can assume that \(M \times N\) is endowed with a product metric.
  By the product formula for the index class (compare \cref{AppendixInj}) and injectivity of \labelcref{ExtProductMap}, we have \(V \times N \in \mathcal{KO}\) for every \(V \in \mathcal{KO}\).
  For a product metric on \(g \oplus h\) on \(V \times N\), the equality \(\width(V \times N, g\oplus h) = \width(V, g)\) holds.
  Thus, if \(b \colon V \looparrowright M\) is a band of width \(\geq l\), then \(b \times \operatorname{id} \colon V \times N \looparrowright M \times N\) is also a band of width \(\geq l\), so the result follows.
\end{proof}

In the Appendix \labelcref{AppendixInj}, we provide a sufficient condition for the injectivity of \labelcref{ExtProductMap} for the reduced group \textCstar-algebra which includes all non-positively curved manifolds \(N\).
In particular, it applies to \(N = \Sphere^1\) which was already used in the proof of \cref{QuadraticDecay}.

Another simple case are Bott manifolds \(N = B\).
This is a simply connected \(8\)-dimensional spin manifold such that its \(\alpha\)-invariant is the Bott generator.
Then \labelcref{ExtProductMap} is the Bott periodicity isomorphism and in particular injective.
Thus, if \(M\) has infinite \(\mathcal{KO}\)-width, so does \(M \times B^k\) for every \(k \geq 0\).
A closed manifold \(M\) is said to \emph{stably} admit a metric of positive scalar curvature, if \(M \times B^k\) admits a metric of positive scalar curvature for some \(k \geq 0\).
Hence we obtain the following strengthening of the obstruction from \cref{InfiniteKOWidthObstruction}.
\begin{cor}
  Let \(M\) be a closed manifold of infinite \(\mathcal{KO}\)-width.
  Then \(M\) does not stably admit a metric of positive scalar curvature.
\end{cor}
The stable Gromov–Lawson–Rosenberg conjecture~\cite[Conjecture~4.17]{RosenbergStolz:PSCSurgeryConnections} predicts that a spin manifold stably admits a metric of positive scalar curvature if and only if its Rosenberg index vanishes.
Unlike the unstable conjecture, this is known for a very large class of fundamental groups, and no counterexample is known.
Indeed, Stolz~\cite{Stolz:ManifoldsOfPSC} proved that the stable conjecture holds if the real Baum--Connes assembly map of \(\pi_1 M\) is injective.
Together with the meta-conjecture~\cite[Conjecture~1.5]{Schick:ICM}, this motivates the following.
\begin{conj}
  Every closed spin manifold of infinite \(\mathcal{KO}\)-width has non-vanishing Rosenberg index.
\end{conj}
Note that for all the examples we mention above this is the case.
A slightly weaker question would be to ask the same for infinite \(\mathcal{K}\)-width or \(\mathcal{K}_{\Q}\)-width.

\appendix

\section{The partitioned manifold index theorem}\label{AppendixPMI}
In this appendix, we add more detail to the brief discussion from \cref{sec:TechnicalTheorem} on the partitioned manifold index theorem for Callias-type operators.
Note that the following theorem is already known in various different guises.
For instance, \citeauthor{Cecchini:CalliasTypePSC}~\cite{Cecchini:CalliasTypePSC} provides the same statement for complex K-theory.
The result is also implicit in the recent work of \citeauthor{Ebert:IndexTheorySpaces}~\cite{Ebert:IndexTheorySpaces}, where a family version of this index theory is developed.
In spirit, this approach to the partitioned manifold index theorem goes back to Higson~\cite{Higson:CobordismInvariance} and Bunke~\cite{Bunke:RelativeIndexCallias}.

\begin{thm}\label{PartitionedManifoldTheorem}
  Let \(W^n\) be a complete spin manifold together with a proper smooth map \(x \colon W \to \R\) with \(\|\nabla x\|_\infty \leq L < \infty\).
  Let \(A\) be some Real \textCstar-algebra and let \(E \to W\) be a smooth bundle of finitely generated projective Hilbert \(A\)-modules \(E\) furnished with a metric connection.
  Let \(\Dirac_{W,E}\) denote the spinor Dirac operator of \(W\) twisted by \(E\).

  Then the operator
  \[
    B = \Dirac_{W, E} \tensgr 1 + x \tensgr \epsilon
  \]
  with initial domain the compactly supported smooth sections of \(\mathcal{E} \coloneqq \SpinBdl_W \tensgr E \tensgr \Cl_{0,1}\), viewed as a bundle of of graded Hilbert \(\mathcal{A} \coloneqq \Cl_{n,0} \tensgr A \tensgr \Cl_{0,1}\)-modules, is an essentially self-adjoint regular operator with compact resolvents.
  In particular, \(B\) has a Fredholm index \(\ind_{\PM}(\Dirac_{W,E}, x) \coloneqq \ind(B) \in \KO_0(\mathcal{A}) = \KO_{n-1}(A)\).

  Moreover, this index has the following properties:
  \begin{myenumi}
    \item\label{item:coarseEq} If \(\tilde{x} \colon W \to \R\) is another proper smooth functions with bounded gradient such that \(x - \tilde{x} \in \ELL^\infty\), then \(\ind_{\PM}(\Dirac_{W,E}, \tilde{x}) = \ind_{\PM}(\Dirac_{W,E}, x)\).
    \item\label{item:PMIT} Let \(a \in \R\) be a regular value of \(x\) and \(M \coloneqq x^{-1}(a)\).
    Then
    \[
      \ind_{\PM}(\Dirac_{W,E}, x) = \ind(\Dirac_{M, E|_{M}}) \in \KO_{n-1}(A).
    \]
  \end{myenumi}
\end{thm}
The proof we provide below follows the standard strategy of cutting and pasting to reduce it to the product situation.
However, it is different from the existing literature in that we use the spectral picture of \(\K\)-theory and asymptotic morphisms to define and manipulate the index.
This is similar to the approach the author took in \cite{Zeidler:PositiveScalarCurvature}.
The spectral picture of graded K-theory goes back to Trout~\cite{Trout:GradedKTheory} and can be viewed as simplified special case of E-theory. For background material, see for instance \cite[Chapters~1--2]{Higson:GroupAstAlgebras}, \cite[Section~2.9]{YuWillett:HigherIndexTheory}.
The spectral picture works for Real \textCstar-algebras without modification and is well-suited to describe index classes of Dirac-type operators, in particular taking Clifford-algebra coefficients and gradings into account.
For our purposes, the main observation is that if \(B\) is an odd regular self-adjoint operator with compact resolvents on some graded Hilbert \(\mathcal{A}\)-module \(\mathcal{X}\), then the functional calculus yields a grading-preserving \(\ast\)-homomorphism \(\Cgr \to \Kom_{\mathcal{A}}(\mathcal{X})\), \(f \mapsto f(B)\), where \(\Kom_{\mathcal{A}}(\mathcal{X})\) denotes the compact operators on \(\mathcal{X}\) in the sense of Hilbert \(A\)-modules.
Here \(\Cgr\) denotes the \textCstar-algebra \(\Cz(\R)\) endowed with the grading defined by the decomposition into even and odd functions.
Such a \(\ast\)-homomorphism represents a class in \(\KO_0(\mathcal{A})\) using the spectral picture, and this is the index of \(B\).
Moreover, in the spectral picture, the external product \(\KO_p(\mathcal{A}_1) \otimes \KO_q(\mathcal{A}_2) \to \KO_{p+q}(\mathcal{A}_1 \otimes \mathcal{A}_2)\) can be constructed  in such a way that it is evident that the product of the indices of suitable operators \(B_1\) and \(B_2\) is equal to the index of \(B_1 \tensgr 1 + 1 \tensgr B_2\), compare~\cite[Section~1.7]{Higson:GroupAstAlgebras}.

\begin{proof}[Proof of \cref{PartitionedManifoldTheorem}]
The differential operator \(B\) is a symmetric because the Dirac operator is and \(\epsilon^\ast = \epsilon \in \Cl_{0,1}\).
In the following, we will make no distinction between \(B\) with its initial domain and its closure.
By \cite[Theorem~1.14]{Ebert:EllipticRegularityDirac}, \(B\) is self-adjoint and regular (use for instance \(\sqrt{1+x^2}\) as a coercive function).
In particular, we have a functional calculus \(\Cz(\R) \to \Lin_{\mathcal{A}}(\ELL^2 \mathcal{E})\), \(f \mapsto f(B)\).
We have the formula
\begin{equation}
  B^2 = \Dirac_{W, E}^2 \tensgr 1 + \clm(\D x) \tensgr \epsilon + x^2 \label{AppendixSquareFormula}
\end{equation}
from which we deduce \(B^2 \geq x^2 - L\).
Hence (the proof of) \cite[Theorem~2.40]{Ebert:EllipticRegularityDirac} shows that \(B\) has compact resolvents, that is, \((B \pm \iu)^{-1} \in \Kom_{\mathcal{A}}(\ELL^2 \mathcal{E})\).
Since \(\Cz(\R)\) is generated by \((\mathrm{x}\pm\iu)^{-1}\) as a \textCstar-algebra, we deduce that \(f(B) \in \Kom_{\mathcal{A}}(\ELL^2 \mathcal{E})\) for every \(f \in \Cz(\R)\).
By the previous discussion, we thus obtain a graded \(\ast\)-homomorphism
\[
  \Phi \colon \Cgr \to  \Kom_{\mathcal{A}}(\ELL^2 \mathcal{E}),\quad f \mapsto f(B)
\]
which represents \(\ind_{\PM}(\Dirac_{W,E}, x) \coloneqq \ind(B) \coloneqq [\Phi] \in \KO_0(\mathcal{A})\) in the spectral picture of K-theory.

It will be convenient to allow more flexibility by working with an asymptotic family of \(\ast\)-homomorphisms instead of a single \(\ast\)-homomorphism.
Consider for each real number \(s \geq 1\) the operator
\begin{equation}\label{OperatorFamily}
  B_s = \frac{1}{s} \left( \Dirac_{W,E} \tensgr 1 + x \tensgr \epsilon\right) + (s-1) \tensgr \epsilon.
\end{equation}
Each \(B_s\) is an operator of the same type as the original operator \(B = B_1\) and we have  \(f(B_s) \in \Kom_{\mathcal{A}}(\ELL^2\mathcal{E})\) for each \(s \geq 1\), \(f \in \Cgr\).
Note that the corresponding version of \labelcref{AppendixSquareFormula} now yields the estimate
\begin{equation}
  B^2_s \geq \left( \frac{x}{s} + s - 1 \right)^2 - \frac{L}{s^2}. \label{BAsymptoticLowerBound}
\end{equation}
\begin{sloppypar}
Then \( \Phi_{s} \colon \Cgr \to  \Kom_{\mathcal{A}}(\ELL^2 \mathcal{E}), f \mapsto f(B_s)\) is a continuous family of \(\ast\)-homomorphisms.
In particular, we obtain an asymptotic morphism \(\Phi_{s} \colon \Cgr \dashrightarrow \Kom_{\mathcal{A}}(\ELL^2 \mathcal{E})\).
The class in E-theory represented by the asymptotic morphism \(\Phi_s\) is the same as the one represented by \(\Phi_1 = \Phi\) in the spectral picture.
\end{sloppypar}

To prove \labelcref{item:coarseEq}, let \(\tilde{B}_s\) and \(\tilde{\Phi}_s\) be the objects defined analogously with \(\tilde{x}\) instead of \(x\).
Then \(B_s - \tilde{B}_s = s^{-1} (x - \tilde{x}) \tensgr \epsilon\) is by assumption a bounded operator which goes to \(0\) in norm as \(s \to \infty\).
This implies that \(\Phi_s\) and \(\tilde{\Phi}_s\) are asymptotically equivalent and, in turn, that \(\ind_{\PM}(\Dirac_{W,E}, x)= [\Phi_s] = [\tilde{\Phi}_s] = \ind_{\PM}(\Dirac_{W,E}, \tilde{x})\).

The proof of \labelcref{item:PMIT} follows the standard strategy of reducing the problem to the cylinder by a cutting and pasting argument.

Start with the case \(W = \R\) and \(x\) the identity map and consider the operator \(B_{\R} = \Dirac_{\R} + x \epsilon\) acting on \(\ELL^2(\R, \Cl_{1,1})\), where \(\Dirac_{\R} = e_1 \frac{\D}{\D x}\).
The index class \(\ind_{\PM}(\Dirac_{\R}, x) = \ind(B_{\R})\) is equal to the generator \(1 \in \KO_0(\R)\).
This follows from a standard computation of the spectrum of the Harmonic oscilator \(H = - \frac{\D^2}{\D x^2} + x^2 - 1\).
See for instance~\cite[Section~1.13]{Higson:GroupAstAlgebras}, where this is worked out in this context.

More generally, consider the product situation \(W = M \times \R\), where \(M\) is a closed spin manifold, \(x\) is the projection on the second factor, and \(E = x^{\ast} E_M\) for some bundle \(E_M \to M\) of finitely generated projective Hilbert \(A\)-modules with a metric connection.
Then \(\Dirac_{W,E} =  \Dirac_{M, E_M} \tensgr 1 + 1 \tensgr \Dirac_{\R} \), where we decompose the bundle \(\SpinBdl_W \tensgr E\) as \((\SpinBdl_M \tensgr E_M) \boxtensgr \SpinBdl_{\R}\) .
From this it follows that the corresponding operator \(B\) can be rewritten as
\[
 B = \Dirac_{M, E_M} \tensgr 1 + 1 \tensgr B_{\R}.
\]
It follows that the index class of \(B\) is equal to the exterior product of the index classes of \(\Dirac_{M, E_M}\) and \(B_{\R}\).
Consequently, we obtain
\begin{equation}\label{eq:ProductCase}
  \ind_{\PM}(\Dirac_{W,E}, x) = \ind(\Dirac_{M, E_M}) \times \ind(B_{\R}) = \ind(\Dirac_{M, E_M}),
\end{equation}
where we used that \(\ind(B_{\R}) = 1\).
Hence \labelcref{item:PMIT} holds for the product case.

To reduce the general case to the product case, we use the following cutting and pasting lemma.
\begin{lem}\label{CutPasteLemma}
  For \(i \in \{1,2\}\), let \({W_i}\) be a complete \(n\)-dimensional spin manifold with a proper smooth map \(x_i \colon W_i \to \R\) with uniformly bounded gradient, \(E_i \to W_i\) a smooth bundle of finitely generated projective Hilbert \(A\)-modules furnished with a metric connection.
  Let \(a \in \R\) be a regular value of both \(x_i\) and set \(W_i^{\leq a} = x^{-1}_i((-\infty, a])\).
  Suppose that there exists an isometric diffeomorphism \(\gamma \colon W_2^{\leq a} \to W_1^{\leq a}\) which is covered by an isometry of the respective spinor bundles and a bundle isometry \(E_2|_{W_2^{\leq a}} \to E_1|_{W_1^{\leq a}}\).
  Furthermore, we assume that \(x_1 \circ \gamma = x_2\) on \(W_2^{\leq a}\).
  Then
  \[
    \ind_{\PM}(\Dirac_{W_1, E_1}, x_1) = \ind_{\PM}(\Dirac_{W_2, E_2}, x_2).
  \]

  The analogous variant of the statement for \(W_i^{\geq a} = x^{-1}_i([a, \infty))\) also holds.
  \end{lem}
  \begin{proof}[Proof of the lemma]
    Using the Kasparov stabilization theorem, we view the Hilbert \(\mathcal{A}\)-module \(\ELL^2(\mathcal{E}_1)\) as a complemented submodule of the standard module \(\ell^2 \mathcal{A}\).
    We can also arrange it in such a way that its orthogonal complement contains another copy of \(\ell^2 \mathcal{A}\).
    The diffeomorphism \(\gamma\) and the corresponding bundle isometries induce an isometry
    \[V^{\leq a} \colon \ELL^2(\mathcal{E}_2|_{W_2^{\leq a}}) \xrightarrow{\cong} \ELL^2(\mathcal{E}_1|_{W_1^{\leq a}}) \subset \ELL^2(\mathcal{E}_1) \subset \ell^2\mathcal{A}.\]
    Then, since \(\ELL^2(\mathcal{E}_2|_{W_2^{\leq a}})\) is a complemented submodule of \(\ELL^2(\mathcal{E}_2)\), by Kasparov stabilization there exists an extension of \(V^{\leq a}\) to an isometry \(V \colon \ELL^2(\mathcal{E}_2) \hookrightarrow \ell^2 \mathcal{A}\).
    Furthermore, let \(B_{i,s}\) be the operator as in \labelcref{OperatorFamily} corresponding to \(W_i\), \(x_i\), \(E_i\), and \(\Phi_{i, s} \colon \Cgr \dashrightarrow \Kom_{\mathcal{A}}(\ELL^2(\mathcal{E}_i)), f \mapsto f(B_{i,s})\) the associated asymptotic morphism.
    The lemma is proved once we show that we have an asymptotic equivalence
    \begin{equation}
      \Phi_{1, s} \sim V \Phi_{2, s} V^\ast \colon \Cgr \dashrightarrow \Kom_{\mathcal{A}}(\ell^2 \mathcal{A}). \label{CutPasteEquivalence}
    \end{equation}
    Here we implicitly use the corner inclusion \(\Lin_{\mathcal{A}}(\ELL^2(\mathcal{E}_1)) \subset \Lin_{\mathcal{A}}(\ell^2 \mathcal{A})\).

    To prove \labelcref{CutPasteEquivalence}, let \(P_i^{\leq b} \in \Lin_{\mathcal{A}}(\ELL^2(\mathcal{E}_i))\) be the orthogonal projection onto \(\ELL^2(\mathcal{E}_i|_{W_i^{\leq b}})\).
    The crucial observation which does the main work is that for each \(b \in \R\) and \(f \in \Cgr\), we have
    \begin{equation}
      \lim_{s \to \infty} \|f(B_{i,s}) (1 - P_i^{\leq b})\| = 0.
      \label{AsymptoticHalf}
    \end{equation}
    This is because for \(u\) in the domain of \(B_i^2\) with \(\supp(u) \subseteq x_i^{-1}((c, \infty))\) for some fixed \(c \in \R\), it follows from \labelcref{BAsymptoticLowerBound} that the estimate
    \[
    \langle B_{i, s}^2 u \mid u \rangle \geq  \left(s - 1 + \frac{c}{s}\right)^2 \langle u \mid u \rangle \geq  \left(s - 2\right)^2 \langle u \mid u \rangle
    \]
    holds for sufficiently large \(s \gg 1\) independently of \(u\).
    In other words, on any region of the form \(x_i^{-1}((c, \infty))\) the operator \(B_{i,s}\) is eventually bounded from below by an arbitrarily large constant.
    If \(f\) is compactly supported, then \labelcref{AsymptoticHalf} follows from this fact by the same argument as in the proof of  \cite[Proposition~3.15]{HankePapeSchick:CodimensionTwoIndex}.
    For general \(f \in \Cgr\) it then follows by approximation.

    We also claim that for every \(b < a\), we have an asymptotic equivalence
    \begin{equation}\label{SwapEquivalence}
      \Phi_{1,s} V P_2^{\leq b} \sim V \Phi_{2,s} P_2^{\leq b}.
    \end{equation}
    Note that the propagation speed of the wave equation associated to the differential operator \(B_{i,s}\) is \(s^{-1}\) and hence goes to zero as \(s \to \infty\).
    So, if \(f \in \Cgr\) has compactly supported Fourier transform, then a standard Fourier theory argument for the operators \(B_{i,s}\) shows that there exists \(s_0 \geq 1\) such that for each \(s \geq s_0\) and \(u \in \ELL^2(\mathcal{E}_2|_{W_2^{\leq b}})\), we have \(f(B_{2,s}) u \in \ELL^2(\mathcal{E}_2|_{W_2^{\leq a}})\) and \(V^{\leq a} f(B_{2,s}) u = f(B_{1,s}) V^{\leq a} u\).
    This proves that for each fixed function \(f \in \Cgr\) with compactly supported Fourier transform, \labelcref{SwapEquivalence} is an equality for \(s \gg 1\).
    By approximation the asymptotic equivalence \labelcref{SwapEquivalence} follows.

    Note that by construction of \(V\), for every \(b \leq a\), the equality
    \begin{equation}\label{SwapTheP}
      P_1^{\leq b} V = V P_2^{\leq b}
    \end{equation}
    holds.
    We are now ready to prove \labelcref{CutPasteEquivalence} and thereby finish the proof of the lemma.
    Choose any \(b < a\).
    Then
    \begin{align*}
      \Phi_{1,s}
       \underset{\labelcref{AsymptoticHalf}}{\sim} \Phi_{1,s} P_1^{\leq b}
       = \Phi_{1,s} P_1^{\leq b} V V^\ast
       &\underset{\labelcref{SwapTheP}}{=} \Phi_{1,s} V P_2^{\leq b} V^\ast \\
       &\underset{\labelcref{SwapEquivalence}}{\sim} V \Phi_{2,s} P_2^{\leq b} V^\ast
       \underset{\labelcref{AsymptoticHalf}}{\sim} V \Phi_{2,s} V^\ast. \qedhere
    \end{align*}
  \end{proof}

  We are now ready to finish the proof of \Cref{PartitionedManifoldTheorem}\labelcref{item:PMIT}.
  Indeed, let \(M \times (-3,3)\) be a tubular neighborhood of \(M = x^{-1}(a)\) in \(W\).
  Using \labelcref{item:coarseEq}, we can modify the function \(x\) in such a way that it is just the projection onto the second factor on \(M \times (-2,2)\) without changing the index.
  We now modify \(W\) by cutting out \(x^{-1}([0, \infty))\) and replacing it with \(M \times [0, \infty)\).
  Using a linear interpolation, we can find a metric on the new manifold which has product structure on \(M \times [1, \infty)\) and agrees with the original one on \(x^{-1}((-\infty, 0])\).
  Similarly, we can find a new bundle \(E\)—together with bundle metric and connection—which on \(M \times [1, \infty)\) is just the pullback of the original bundle restricted to \(M\), and on \(x^{-1}((-\infty, 0])\) agrees with the original data.
  By \cref{CutPasteLemma} (applied to \(a = 0\)) this procedure does not change \(\ind_{\PM}(\Dirac_{W,E}, x)\).
  Using the reversed variant of \cref{CutPasteLemma}, we can furthermore replace \(x^{-1}((-\infty, 1])\) by \(M \times (-\infty, 1]\) (with all metric and bundle data being of product form) and still have the same index class.
  Therefore, the original partitioned manifold index \(\ind_{\PM}(\Dirac_{W,E}, x)\) is equal to \(\ind_{\PM}(\Dirac_{M \times \R, x^\ast E|_M}, x)\) with \(x\) being the projection onto the second factor.
  This finishes the proof of the theorem by \labelcref{eq:ProductCase}.
\end{proof}

\section{Injectivity of exterior products}\label{AppendixInj}
For \(i \in \{1,2\}\), let \(M_i\) be a closed \(m_i\)-dimensional spin manifold which is endowed with a flat bundle \(E_i \to M_i\) of finitely generated projective Hilbert \(A_i\)-modules.
On the product \(M_1 \times M_2\), we can form the exterior tensor product bundle \(E_1 \boxtimes E_2\) which is a flat bundle of finitely generated projective Hilbert \(A_1 \otimes A_2\)-modules.
 We use the spacial tensor product of \textCstar-algebras.
 It is a standard fact that in this situation the equality
\begin{equation}
  \ind(\Dirac_{M_1, E_1}) \times \ind(\Dirac_{M_2, E_2}) = \ind(\Dirac_{M_1 \times M_2, E_1 \boxtimes E_2})
  \label{eq:ProductFormula}
\end{equation}
holds in \(\KO_{m_1 + m_2}(A_1 \otimes A_2)\).
For instance, this can be readily verified in the spectral picture of K-theory mentioned in Appendix~\labelcref{AppendixPMI}.

In light of this product formula, it is sometimes important to know that taking the exterior product with the index class of a fixed manifold defines an injective map on K-theory.

A complete \(n\)-dimensional Riemannian manifold \(X\) is called \emph{hypereuclidean} if there exists a proper Lipschitz map \(X \to \R^n\) of degree one.
Moreover, we call \(X\) \emph{stably hypereuclidean} if \(X \times \R^k\) is hypereuclidean for some \(k \geq 0\).

\begin{prop}\label{HypereuclideanInjective}
  Let \(N\) be a closed manifold such that its universal covering \(\tilde{N}\) is stably hypereuclidean.
  Then  for every Real \textCstar-algebra \(A\), the exterior product map induced by the reduced Rosenberg index
  \[
    \KO_\ast(A) \xrightarrow{\alpha_{\red}(N) \times \blank} \KO_{\ast+n}(\CstarRed \pi_1 N \otimes A)
  \]
  is injective.
\end{prop}
In particular, this holds for \(N = \Sphere^1\) or any non-positively curved manifold \(N\).
Moreover, by a result of Dranishnikov~\cite{Dranishniko:Hypereuclidean}, if \(N\) is aspherical and \(\pi_1 N\) has finite asymptotic dimension, then \(\tilde{N}\) is stably hypereuclidean.
The statement of \cref{HypereuclideanInjective} is analogous to \cite[Corollary~5.8]{Zeidler:PositiveScalarCurvature} and follows essentially from the same proof.
Moreover, in~\cite{EngelWulffZeidler_CrossProducts} a more general framework to obtain injectivity statements of this type is provided.

\begin{proof}
  Let \(X = \tilde{N}\).
  We use the \emph{Roe algebra} \(\Roe(X; A)\) with coefficients in a \textCstar-algebra \(A\).
  We concretely construct \(\Roe(X; A)\) on the Hilbert \(A\)-module \(\ELL^2(\SpinBdl) \otimes A\), where \(\ELL^2(\SpinBdl)\) denotes \(\ELL^2\)-sections of the spinor bundle.
  For definitions of the Roe algebra with coefficients, see for instance \cite{HigsonPedesenRoe:Controlled} or \cite[Definition~3.2]{HankePapeSchick:CodimensionTwoIndex}.
  It suffices to consider the case that \(X\) is hypereuclidean.
  Let \(f \colon X \to \R^n\) be a degree one proper Lipschitz map.
  It induces a map \(f_\ast \colon \KO_{\ast}(\Roe (X; A) ) \to \KO_\ast(\Roe(\R^n; A))\) on the K-theory of the Roe algebra for any coefficient \textCstar-algebra.
  Let \(\Lambda = \pi_1 N\).
  The equivariant Roe algebra \(\Roe[\Lambda](X; A)\) can be canonically identified with \(\Roe[\Lambda](X; \R) \otimes A\) and is Morita equivalent to \(\CstarRed \Lambda \otimes A\).
  The latter holds for instance by \cite{Roe:ComparingAnalyticAssembly}.
  In view of this Morita equivalence, the reduced Rosenberg index \(\alpha_{\red}(N) \in \KO_{n}(\CstarRed \Lambda)\) identifies with the equivariant coarse index \(\ind_{\mathrm{c}}^\Lambda(\Dirac_X) \in \KO_n(\Roe[\Lambda](X; \R))\) and it suffices to prove injectivity of the exterior product map induced by the class \(\ind_{\mathrm{c}}^\Lambda(\Dirac_X)\).

  Note that in the non-equivariant case, in general \(\Roe(X;\R) \otimes A \subsetneq \Roe(X; A)\).
  However, on \(\R^n\), we can still consider the following composition
  \[
    \Phi \colon \KO_p(A) \xrightarrow{\ind_{\mathrm{c}}(\Dirac_{\R^n}) \times \blank} \KO_{n+p}(\Roe(\R^n; \R) \otimes A) \to \KO_{n+p}(\Roe(\R^n; A)).
  \]
  The homomorphism \(\Phi\) is an isomorphism.
  The inverse is given by the iterated application of boundary maps in the coarse Mayer--Vietoris sequence (\cite[Corollary~2.11]{SchickZadeh:MultiPart}, \cite[Corollary~9.5]{HigsonPedesenRoe:Controlled}).
  \[\KO_{n+p}(\R^n; A) \xrightarrow{\partial_{\mathrm{MV}}^n} \KO_{p}(\Roe(\{0\}; A)) \cong \KO_{p}(A)\]
  Using the restriction map \(\mathrm{r} \colon \Roe[\Lambda](X;A) \to \Roe(X;A)\) which forgets \(\Lambda\)-invariance, we obtain the following diagram
  \[
    \begin{tikzcd}[column sep=small]
      \KO_p(A) \rar{\ind_{\mathrm{c}}^\Lambda(\Dirac_{X}) \times \blank} \ar[dd, "\ind_{\mathrm{c}}(\Dirac_{\R^n}) \times \blank", swap] \ar[ddrr, "\Phi"', "\cong"] & \KO_{n+p}(\Roe[\Lambda](X;\R) \otimes A) \rar[equal]& \KO_{n+p}(\Roe[\Lambda](X;A)) \dar["\mathrm{r}_\ast"] \\
      & & \KO_{n+p}(\Roe(X;A)) \dar["f_\ast"] \\
      \KO_{n+p}(\Roe(\R^n; \R) \otimes A) \ar[rr] & & \KO_{n+p}(\Roe(\R^n; A)).
    \end{tikzcd}
  \]
  It is commutative  since \(f\) has degree one and hence takes the coarse index class \(\ind_{\mathrm{c}}(\Dirac_{X}) \in \KO_n(\Cstar(X; \R))\) to \(\ind_{\mathrm{c}}(\Dirac_{\R^n}) \in \KO_n(\Cstar(\R^n; \R)) \cong \Z\).
  Since \(\Phi\) is an isomorphism, the desired injectivity follows.
\end{proof}
\begin{rem}
  Another potential way to obtain such an injectivity statement would be to use the Künneth formula—at least if the \textCstar-algebra \(\CstarRed \pi_1 N\) is in a class that admits such a formula.
  Especially in the case \(N = \Sphere^1\) this can be done straightforwardly for complex K-theory, compare~\cite[Proposition~4.2]{HankePapeSchick:CodimensionTwoIndex}.
  However, the Künneth formula in the realm of Real \textCstar-algebras is not as straightforward (see~\cite{Boersema:RealKuenneth}).
  This is why we preferred to exhibit the argument above which proves injectivity directly.
\end{rem}
\printbibliography
\end{document}

%% file: BandCylinder.pdf_tex
\begingroup%
  \makeatletter%
  \providecommand\color[2][]{%
    \errmessage{(Inkscape) Color is used for the text in Inkscape, but the package 'color.sty' is not loaded}%
    \renewcommand\color[2][]{}%
  }%
  \providecommand\transparent[1]{%
    \errmessage{(Inkscape) Transparency is used (non-zero) for the text in Inkscape, but the package 'transparent.sty' is not loaded}%
    \renewcommand\transparent[1]{}%
  }%
  \providecommand\rotatebox[2]{#2}%
  \newcommand*\fsize{\dimexpr\f@size pt\relax}%
  \newcommand*\lineheight[1]{\fontsize{\fsize}{#1\fsize}\selectfont}%
  \ifx\svgwidth\undefined%
    \setlength{\unitlength}{478.79996713bp}%
    \ifx\svgscale\undefined%
      \relax%
    \else%
      \setlength{\unitlength}{\unitlength * \real{\svgscale}}%
    \fi%
  \else%
    \setlength{\unitlength}{\svgwidth}%
  \fi%
  \global\let\svgwidth\undefined%
  \global\let\svgscale\undefined%
  \makeatother%
  \begin{picture}(1,0.34574769)%
    \lineheight{1}%
    \setlength\tabcolsep{0pt}%
    \put(0,0){\includegraphics[width=\unitlength,page=1]{BandCylinder.pdf}}%
    \put(0.24708034,0.00492564){\color[rgb]{0,0,0}\makebox(0,0)[lt]{\lineheight{1.25}\smash{\begin{tabular}[t]{l}$\partial_{-}V$\end{tabular}}}}%
    \put(0.70872541,0.0487847){\color[rgb]{0,0,0}\makebox(0,0)[lt]{\lineheight{1.25}\smash{\begin{tabular}[t]{l}$\partial_{+}V$\end{tabular}}}}%
    \put(0.1792271,0.31651508){\color[rgb]{0,0,0}\makebox(0,0)[lt]{\lineheight{1.25}\smash{\begin{tabular}[t]{l}$W_{-}$\end{tabular}}}}%
    \put(0.3126547,0.31651508){\color[rgb]{0,0,0}\makebox(0,0)[lt]{\lineheight{1.25}\smash{\begin{tabular}[t]{l}$W_{+}$\end{tabular}}}}%
    \put(0.47956812,0.05754403){\color[rgb]{0,0,0}\makebox(0,0)[lt]{\lineheight{1.25}\smash{\begin{tabular}[t]{l}$V$\end{tabular}}}}%
    \put(0,0){\includegraphics[width=\unitlength,page=2]{BandCylinder.pdf}}%
  \end{picture}%
\endgroup%